\documentclass[11pt]{amsart}

\usepackage[left=2.5cm, right=2.5cm, top=2.5cm, bottom=2.5cm, bindingoffset=0cm]{geometry}
\usepackage{amsthm, amssymb, amsmath, amsfonts, amsaddr}

\usepackage{hyperref}

\usepackage{graphicx}
\usepackage{pgf, tikz, pgfplots, tikz-cd}
\pgfplotsset{compat=1.15}
\usetikzlibrary{arrows}

\usepackage{mathrsfs}
\usepackage{dynkin-diagrams, ytableau}
\usepackage{multirow}

\allowdisplaybreaks
\usepackage{enumitem}

\newcommand{\rge}{\rangle}
\newcommand{\lge}{\langle}

\newcommand{\mc}{\mathcal}

\newcommand{\ket}[1]{|#1\rangle}

\newcommand{\ch}{\operatorname{ch}}

\renewcommand{\le}{\leqslant}
\renewcommand{\ge}{\geqslant}

\newcommand{\C}{\mathbb{C}}

\newcommand{\Z}{\mathbb{Z}}

\newcommand{\PP}{\mathbb{P}}

\renewcommand{\b}{\mathfrak{b}}

\newcommand{\vv}{\mathfrak{v}}

\renewcommand{\r}{\mathfrak{r}}
\newcommand{\m}{\mathfrak{m}}
\newcommand{\n}{\mathfrak{n}}
\newcommand{\g}{\mathfrak{g}}
\newcommand{\h}{\mathfrak{h}}

\newcommand{\Vir}{\mathrm{Vir}}
\newcommand{\Heis}{\mathrm{Heis}}

\renewcommand{\sl}{\mathfrak{sl}}

\newcommand{\mat}[4]{\begin{pmatrix} #1 & #2 \\ #3 & #4 \end{pmatrix}}

\newcommand{\jet}[1]{[\hspace{-0.5mm}[{#1}]\hspace{-0.5mm}]}

\numberwithin{equation}{section}
\theoremstyle{plain}
\newtheorem{theorem}{Theorem}[section]
\newtheorem{definition}[theorem]{Definition}
\newtheorem{proposition}[theorem]{Proposition}
\newtheorem{corollary}[theorem]{Corollary}
\newtheorem{lemma}[theorem]{Lemma}

\newtheorem{conjecture}[theorem]{Conjecture}
\theoremstyle{definition}
\newtheorem{remark}[theorem]{Remark}
\newtheorem{computation}[theorem]{Computation}
\newtheorem{example}[theorem]{Example}
\newtheorem{antiexample}[theorem]{Anti-example}

\title[Virasoro extensions of generalized Kac-Moody algebras]{Virasoro extensions of generalized Kac-Moody algebras}
\author[Vladyslav Zveryk]{Vladyslav Zveryk}
\address{Yale University}
\email{zverik.vladislav@gmail.com}

\begin{document}

\begin{abstract}
    Borcherds-Kac-Moody (BKM) Lie algebras are constructed by combining a number of $\sl_2$-s and $3$-dimensional Heisenberg algebras in a certain way, generalizing the celebrated Kac-Moody algebras. Such algebras have a wide range of applications, including the examples of the monster and fake monster Lie algebras being BKM algebras. Among their main features are existence of an invariant bilinear form, denominator identity and Weyl-Kac character formulas for irreducible representations.
    
    We introduce a construction of closely related algebras that we call Virasoro-Borcherds-Kac-Moody (VBKM) algebras: if it happens that certain $3$-dimensional Heisenberg algebras inside a BKM algebra form the infinite-dimensional Heisenberg algebra, we replace it by the Virasoro algebra in the construction. We construct a family of Lie algebras over $\PP^1$, where the generic fiber is a VBKM algebra and a special fiber is a BKM algebra. This allows us to visualize VBKM algebras as deformations of BKM algebras and prove structural properties about VBKM algebras similar to those of BKM algebras, despite the absence of an invariant bilinear form. Those include denominator identity and character formulas of irreducible representations.
    
    We give examples of realizations of such $\PP^1$-families of VBKM algebras in lattice vertex operator algebras. In particular, we include the fake monster Lie algebra into such a family. 
\end{abstract}

\maketitle
\tableofcontents

\section{Introduction}
Borcherds-Kac-Moody algebras were introduced by Borcherds in \cite{Bor88}, where they were called generalized Kac-Moody algebras. They are generalizations of Kac-Moody algebras \cite{Kac90}. Kac-Moody algebras, on their own, aim to generalize simple Lie algebras and their representation theory. For a detailed historical survey on Kac-Moody algebras, we refer to \cite{FT21}.

Kac-Moody algebras are constructed from a square matrix $A=(a_{ij})_{i,j\in I}$ called the Cartan matrix. This matrix has to satisfy a number of key properties, namely:
\begin{itemize}
    \item $a_{ii}=2$ and $a_{ij}\in\Z_{\le 0}$ for $i\ne j$,
    \item $a_{ij}=0$ if and only if $a_{ji}=0$.
\end{itemize}
From this matrix, a Lie algebra $\g$ is constructed: it is generated by $\sl_2$-triples $e_i,h_i,f_i$ with the following commutativity relations between them:
\begin{itemize}
    \item $[e_i,f_j]=\delta_{ij}h_i$,
    \item $[h_i,e_j]=a_{ij}e_j$, $[h_i,f_j]=-a_{ij}f_j$,
    \item $(\operatorname{ad}e_i)^{1-a_{ij}}e_j=(\operatorname{ad}f_i)^{1-a_{ij}}f_j$.
\end{itemize}

When the matrix $A$ is symmetrizable, i.e. it can be made into a symmetric matrix by rescaling its rows, the algebra $\g$ acquires a lot of structural theory. Among it (\cite{Kac90}):
\begin{enumerate}
    \item Classification of Kac-Moody algebras into three types: finite, affine, and indefinite. Complete classification of finite and affine types.
    \item The Cartan decomposition
    $$
    \g=\n_-\oplus\h\oplus\n_+,
    $$
    where $\n_+$ ($\n_-$) is generated by $e_i$ ($f_i$) and $\h$ is generated by $h_i$.
    \item A non-degenerate invariant bilinear form on $\g$ after a suitable extension of its subalgebra $\h$.
    \item Presence of a root datum, hence grading of $\g$ by elements of the root lattice and of $\g$-representations by the weight lattice.
    \item Theory of highest weight representations. In particular, denominator identity for $\g$ and Weyl-Kac character formula for dominant integrable representations.
\end{enumerate}
And the list goes on. One can ask: can we classify Lie algebras possessing properties (2)-(4), and can we develop their representation theory analogous to (5)? 

Both questions were answered by Borcherds in \cite{Bor88} and investigated further in \cite{Jur96,Jur98}. It turns out that one has to relax conditions on the entries of the Cartan matrix, and one still gets an algebra satisfying properties (2)-(5). Algebras obtained this way are called Borcherds-Kac-Moody (BKM) algebras.

We refer to Section \ref{s:BKM_algebras} for precise definition, but we mention one relaxation of particular interest to us. It allows $a_{ii}=0$, in which case the generating triple is assumed to form the $3$-dimensional Heisenberg algebra:
$$
[e_i,f_i]=h_i,\qquad [h_i,e_i]=[h_i,f_i]=0.
$$

One of the most crucial inputs by Borcherds was essentially that whenever one has an algebra satisfying properties (2)-(4), this algebra is a BKM algebra (for a precise statement with all assumptions, see \cite{Bor88}), and that these algebras have a denominator identity, allowing one to get insight into dimensions of their root spaces. One particular source of such algebras are lattice vertex operator algebras. In short, if the rank of a lattice is $26$, the first BRST cohomology of the corresponding lattice VOA becomes a BKM algebra (with a proper replacement of BRST cohomology in smaller ranks, we can get the same statement there, see \cite{BGGN}). This Lie algebra has well-understood root space dimensions, but its Cartan matrix is usually hard to describe. 

However, in rank $26$, \cite{FLM89} and \cite{Bor90} constructed two examples of such lattice VOAs in rank $26$ with their BRST cohomology having its Cartan matrix well-understood. These algebras are called fake monster and monster algebras, respectively. Final proofs and descriptions of Cartan matrices were completed in \cite{Bor90} and \cite{Bor92}. There are other well-understood examples of BKM algebras, see e.g. \cite{Nie02}.

Example of the fake monster Lie algebra will be especially interesting to us. Its simple roots (roughly, triples $e_i,h_i,f_i$) consist of two classes:
\begin{itemize}
    \item Real simple roots generating a Kac-Moody algebra whose Dynkin diagram is isometric with the Leech lattice.
    \item Isotropic simple roots: for each $n\in\Z_{>0}$, $24$ copies of a root $n\rho$ with $a_{ii}=0$ for $\rho$.
\end{itemize}

Isotropic roots combine into a rank $24$ infinite Heisenberg algebra, which itself is formed from $24$ rank one Heisenberg algebras commuting with each other and glued along their central elements. Recall that the infinite-dimensional Heisenberg algebra $\Heis$ is an algebra with basis $x_n$ for $n\in\Z\setminus 0$ and $c$ satisfying relations
$$
[x_n,x_m]=n\delta_{m,-n}c.
$$

One can ask about other algebras appearing inside lattice VOAs. However, when we leave the space giving the BRST cohomology, we lose an invariant bilinear form, thus the entire theory of BKM algebras. Our main idea to overcome this is the following: deform BKM algebras inside the lattice VOAs, and show that the Lie algebras in this deformation family have properties similar to those of BKM algebras, e.g. (2), (4), and (5). This leads us to the definition of Virasoro-Borcherds-Kac-Moody (VBKM) Lie algebras.

First, we observe that the Virasoro algebra $\Vir$ has a similar structure to the infinite-dimensional Heisenberg algebra: it is an algebra with basis $L_n$ for $n\in\Z$ and $c$ and commutation relations
$$
[L_n,L_m]=(n-m)L_{n+m}+\delta_{n,-m}\frac{n^3-n}{12}c.
$$

Then we notice that the Virasoro algebra can be realized as a deformation of the Heisenberg algebra extended by a derivation $L_0$: we let $\Vir$ act on $\Heis$ by
$$
[L_n,x_m]:=-mx_{m+n}.
$$
Then we form the semidirect product $\vv:=\Vir\ltimes\Heis$ and notice that the elements 
$$
Y_n(s,t):=sL_n+tx_n\in\vv
$$
form a basis of the Virasoro algebra, thus interpolating between the Virasoro and the Heisenberg algebra. This works perfectly with BKM algebras like the fake monster Lie algebra where the infinite-dimensional Heisenberg algebras are part of the Cartan matrix. In our definition of VBKM algebras, we replace these Heisenberg algebras with Virasoro algebras and develop a structural theory similar to the theory of BKM algebras.

\subsection{Main results}
Our first main result is the abstract construction itself. We define an enhanced generalized Cartan datum, which is a Cartan matrix for a BKM algebra along with additional rows and columns corresponding to Virasoro algebras that we wish to add into the list of generators. In this case, we not only have generators $e_i,h_i,f_i$, but also $L_n^j$ for $n\in\Z$ and $c^j$.
\begin{theorem}
    Let $\g$ be a VBKM algebra. We have
    \begin{enumerate}[label=(\roman*)]
        \item $\g=\n_-\oplus\h\oplus\n_+$ where $\n_+$ ($\n_-$) is generated by $e_i$ and $\Vir_{>0}^j$ ($f_i$ and $\Vir_{<0}^j$) and $\h$ is spanned by $h_i$, $L_0^j$ and $c^j$.
        \item $\g$ is graded by the root lattice and the root spaces are finite-dimensional. 
        \item The Virasoro algebra and BKM algebras are particular examples of VBKM algebras.
    \end{enumerate}
\end{theorem}
See Section \ref{s:VBKM_algebras}.

Our second main result is in developing the point of view on VBKM algebras as deformations of BKM algebras. First of all, we construct the ultimate algebra $\g_{\Vir,\Heis}$ where we put the algebra $\vv=\Vir\ltimes\Heis$ in the generators instead of $\Vir$. It satisfies analogous properties to those of VBKM algebras described above, the proof of which mimics the VBKM case and is given in the Appendix. 

We construct a family of Lie algebras over the projective line $\PP^1$ contained in $\g_{\Vir,\Heis}$ such that the special fiber is a BKM algebra $\g_{\Heis}$, while any other fiber is a VBKM algebra.
Let $\g_{\Vir}$ be one such fiber. We prove the following.
\begin{theorem}\,
    \begin{enumerate}[label=(\roman*)]
        \item Let $\n_{\Vir,+}$ and $\n_{\Heis,+}$ be the positive parts of $\g_\Vir$ and $\g_\Heis$, respectively. There is a filtration on $\n_{\Vir,+}$ such that the associated graded
        $$  
        \operatorname{gr}\n_{\Vir,+}\simeq \n_{\Heis,+}.
        $$
        In particular, the root spaces are equally-dimensional and the VBKM algebra $\g_{\Vir}$ satisfies the denominator identity of the BKM algebra $\g_{\Heis}$.
        \item Over a generic point of $\PP^1$, $\g_{\Vir}$ has no nontrivial ideals not intersecting $\h$.
        \item Let $\lambda$ be a dominant integral weight and  $L_\Vir(\lambda)$, $L_\Heis(\lambda)$ be the quotients of the corresponding Verma modules for $\g_{\Vir}$ and $\g_{\Heis}$ by Serre relations. Then there is a filtration on $L_\Vir(\lambda)$ such that
        $$
        \operatorname{gr} L_\Vir(\lambda)\simeq L_\Heis(\lambda).
        $$
        \item Over a generic point of $\PP^1$, $L_\Vir(\lambda)$ is irreducible.
    \end{enumerate}
\end{theorem}

For these statements, see Theorems \ref{t:characters_of_g_in_families}, \ref{t:characters_of_reps_in_families}, and \ref{t:ass_graded_of_Vir_is_Heis}.

The third important part in our paper is examples and computations in lattice VOAs developed in Section \ref{s:examples}. In particular, we give a detailed analysis of Virasoro algebras lying in lattice VOAs and the fake monster Lie algebra.

\subsection{Structure of the paper}
In Section 2, we give preliminaries on BKM algebras.

Section 3 is dedicated to constructing VBKM algebras. In Section 4, we prove results analogous to those stated about BKM algebras in Section 2. We construct a family of Lie algebras deforming a BKM algebra into VBKM algebras. We prove the denominator identity and character formulas for VBKM algebras.

In Section 5, we include examples of VBKM algebras realizable inside lattice VOAs. In particular, we describe the case of the fake monster Lie algebra in detail.

In Appendix, we give the proof that the ultimate algebra $\g_{\Vir,\Heis}$ needed for the deformation construction satisfies properties similar to VBKM algebras.

\section{Borcherds-Kac-Moody algebras}\label{s:BKM_algebras}
 For good expositions on the topic, we refer to \cite{Jur96}, \cite{Ray06}, and \cite{Nie02}.

We fix the following data:
\begin{itemize}
    \item $I$ is an index set.
    \item $\h$ and $\h^\vee$ are vector spaces with a bilinear pairing 
    $$
    \lge\cdot,\cdot\rge:\h\otimes\h^\vee\to \C.
    $$
    \item $h_i\in\h$ and $\alpha_i\in\h$ are distinguished elements for $i\in I$ with $\h$ being spanned by $h_i$.
\end{itemize}

\begin{definition}
    The matrix 
$$
A=(a_{ij}):=(\lge h_i,\alpha_j\rge)
$$
is called a {\bfseries generalized Cartan matrix} if
\begin{itemize}
    \item If $a_{ii}>0$, then $2a_{ij}/a_{ii}\in \Z_{\le 0}$ for all $j\in I$. 
    \item For all $i,j\in I$, $a_{ij}=0$ is equivalent to $a_{ji}=0$.
\end{itemize}

We call the entire choice $(\h,\h^\vee,I,\alpha_i,h_i)$ a {\bfseries generalized Cartan datum} if $A$ is a generalized Cartan matrix and for any $\alpha\in\h^\vee$, the number of ways to write
\begin{equation}\label{eq:finiteness_condition}
   \alpha=\sum_{k=1}^mn_{i_k}\alpha_{i_k},\qquad n_{i_k}\in\Z_{\ge 0}
\end{equation}
is finite.

A generalized Cartan datum is called {\bfseries symmetrizable} if there is a linear automorphism of $\h$ acting on $h_i$ by rescaling such that the Cartan matrix with rescaled $h_i$ becomes symmetric.
\end{definition}

The set $I$ then naturally decomposes into three subspaces:
\begin{align*}
    I_{>0}:=\{i\in I:a_{ii}>0\},\\
    I_{0}:=\{i\in I:a_{ii}=0\},\\
    I_{<0}:=\{i\in I:a_{ii}<0\}.
\end{align*}
These spaces will come in handy in our study of BKM Lie algebras.

To a generalized Cartan matrix whose rows and columns are indexed by a set $I$, we can associate a Cartan datum as follows: if $A$ is non-degenerate, we define $\h$ to be a vector space with a chosen basis $h_i$ for $i\in I$. Then $\alpha_i$ are chosen in the dual space $\h^*$ to satisfy $a_{ij}=\alpha_j(h_i)$, and the pairing between these two vector spaces is standard. When $A$ is degenerate, we can find larger spaces $\h$ and $\h^\vee$ in which $h_i$ and $\alpha_i$ are linearly independent. See \cite[Chapter 1]{Kac90} for details. Clearly, condition (\ref{eq:finiteness_condition}) is satisfied when linear independence of $\alpha_i$ is assumed.

Note also that for any subspaces $\m\subset\lge h_i:i\in I\rge\subset  \h$ and $\m^\vee\subset\lge \alpha_i:i\in I\rge\subset \h^\vee$ contained in $\ker A$, the datum
$$
(\h/\m,\h^\vee/\m^\vee,I,\alpha_i,h_i)
$$
is a Cartan datum as soon as the assumption (\ref{eq:finiteness_condition}) is still satisfied, where by $\alpha_i$ ($h_i$) we understand their images modulo $\m^\vee$ ($\m$). See Example \ref{e:Heisenberg_degenerations} for such degenerations.

Let $\mc F$ be a Cartan datum with the notation as above. We define the Lie algebra $\tilde\g(\mc F)$ generated by triples $e_i,h_i,f_i$ for $i\in I$ with the following relations:
\begin{align*}
[h_i,h_j]&=0,\\
[h_i,e_j]&=\alpha_j(h_i)e_j=a_{ij}e_j,\\
[h_i,f_j]&=-\alpha_j(h_i)f_j=-a_{ij}f_j,\\
[e_i,f_j]&=\delta_{ij}h_i.
\end{align*}
We notice the following:
\begin{itemize}
    \item When $i\in I_{>0}$ or $I_{<0}$, the triple $e_i,h_i,f_i$ can be rescaled to an $\sl_2$ triple.
    \item When $i\in I_0$, the triple $e_i,h_i,f_i$ generates the $3$-dimensional Heisenberg algebra.
    \item The vector space $\h$ sits naturally in $\tilde\g(\mc F)$ as an abelian subalgebra.
\end{itemize}

Here are the following facts that we will use. For the proofs, see \cite[Chapter 1]{Kac90} or \cite[Sections 1.1-1.2]{Jur96}.
\begin{theorem}\, \label{t:tilde_g_facts}
    \begin{enumerate}[label=(\roman*)]
        \item The Lie algebra $\tilde \g(\mc F)$ is non-zero.
        \item Let $\n_+$ and $\n_-$ be the subalgebras of $\tilde \g(A)$ generated by $e_i$ and $f_i$ for $i\in I$, respectively. We have a direct sum decomposition
        $$
        \tilde \g(\mc F)=\n_-\oplus\h\oplus\n_+.
        $$
        Moreover, $\h$ acts semisimply on $\tilde \g(\mc F)$.
        \item There exists the largest ideal $\r$ in $\tilde \g(A)$ which does not intersect $\h$. We have 
        $$
        \r=\r\cap \n_+\oplus \r\cap\n_-.
        $$
        If the generalized Cartan datum is symmetrizable, then $\r\cap \n_+$ is generated by {\bfseries Serre relations} 
        \begin{align*}
            (\operatorname{ad}e_i)^{1-2a_{ij}/a_{ii}}e_j&=0,\qquad i\in I_{>0},\\
            [e_i,e_j]&=0,\qquad i,j\in I\text{ such that } a_{ij}=a_{ji}=0.
        \end{align*}
        and $\r\cap\n_-$ being generated by relations 
        \begin{align*}
            (\operatorname{ad}f_i)^{1-2a_{ij}/a_{ii}}f_j&=0,\qquad i\in I_{>0},\\
            [f_i,f_j]&=0,\qquad i,j\in I\text{ such that } a_{ij}=a_{ji}=0.
        \end{align*}
    \end{enumerate}
\end{theorem}

Let us explain where the relations in Theorem \ref{t:tilde_g_facts}(iii) come from and why $\r$ does not intersect $\h$. Pick $i\in I_{>0}\sqcup I_{<0}$. Then the triple 
$$
e,h,f:=\frac{\sqrt2 e_i}{\sqrt{a_{ii}}},\frac{2h_i}{a_{ii}},\frac{\sqrt2 f_i}{\sqrt{a_{ii}}}
$$
is an $\sl_2$-triple. Since $[e_i,f_j]=0$, the vector $f_j\in\tilde\g(\mc F)$ is a highest weight vector for $\sl_2$ with highest weight $\alpha_j(2h_i/{a_{ii}})=2a_{ij}/a_{ii}$ with respect to the adjoint action.

By representation theory of $\sl_2$, the $\sl_2$-submodule generated by this highest weight $f_j$ contains a largest submodule, in which case this submodule has highest weight $v:=(\operatorname{ad}f_i)^{1-2a_{ij}/a_{ii}}f_j$ if $i\in I_{>0}$ and is zero if $i\in I_{<0}$. We have:
\begin{itemize}
    \item $[e_i,v]=0$ by the highest weight condition.
    \item $[e_k,v]=0$ for all $k\ne i,j$ since any such $e_k$ commutes with $f_i$ and $f_j$.
    \item $[e_j,v]=(\operatorname{ad}f_i)^{1-2a_{ij}/a_{ii}}[e_j,f_j]=(\operatorname{ad}f_i)^{-2a_{ij}/a_{ii}}a_{ij}f_j$, which is zero in both possible cases $a_{ij}=0$ and $a_{ij}<0$.
\end{itemize}

Therefore, the $sl_2$-representation defined by $v$ is entirely contained in $\n_-$ and is preserved by $\n_+$. Then the ideal generated by it is generated by $U(\n_-)$ hence stays in $\n_-$ and does not intersect $\h$. Adding the fact that every Verma module for the Heisenberg algebra is irreducible or has highest weight $0$ with the trivial one-dimensional quotient (see \cite[Proposition 1.7.2]{FLM89}), we finish the proof of the following.
\begin{proposition}\label{p:BKM_interpretation_using_reps_of_simple_pieces}
    The quotient $\tilde \g(\mc F)/\r$ is the closest to $\tilde \g(\mc F)$ quotient of $\tilde \g(\mc F)$ such that for any $i$, the subrepresentations of the algebra $\lge e_i,h_i,f_i\rge$ generated by $e_j$ and $f_j$ are irreducible.
\end{proposition}
We will use this interpretation to generalize BKM algebras in subsequent sections. But now, it is time we finally defined BKM algebras.

\begin{definition}
    The Borcherds-Kac-Moody (BKM) algebra associated with a Cartan datum $\mc F$ is $\g(\mc F):=\tilde\g(\mc F)/\r$, where $\r$ is the largest ideal of $\tilde\g(\mc F)$ not intersecting $\h$.
\end{definition}

We can introduce a grading on $\tilde\g(\mc F)$ by $\h^\vee$ and $\g(\mc F)$ by declaring $\deg h_i=0$, $\deg e_i=\alpha_i$ and $\deg f_i=-\alpha_i$. We call $\alpha\in \h^\vee$ {\bfseries a root of $\g(\mc F)$} if $\alpha\ne 0$ and the corresponding {\bfseries root space} $\g(\mc F)_\alpha\ne 0$. Note that condition (\ref{eq:finiteness_condition}) ensures that root spaces are finite-dimensional and $\g(\mc F)_0=\h$.

Let us give some examples of these.
\begin{example}
    The simplest example is Kac-Moody Lie algebras extensively studied in \cite{Kac90}, which is the case $I_{>0}=I$. In particular, when the Cartan matrix is of finite type, we get a finite-dimensional simple Lie algebra.
\end{example}

\begin{example}\label{e:Heisenberg_degenerations}
    In this example, we set $I:=\Z_{>0}$ with $A=0$. Let $\h=\bigoplus_{n>0}\C h_n$, $\h^\vee=\bigoplus_{n>0}\C \alpha_n$. Then $\g=\bigoplus_{n>0} \g_n$ with $\g_n$ commuting with each other and $e_n$ having degree $\alpha_n$. We can have the following degenerations of this:
    \begin{enumerate}[label=(\roman*)]
        \item Let $\h=\C h$ and $\h^\vee=\C\alpha$ with $h_n=h$ and $\alpha_n=n\alpha$. In this case, we obtain the infinite-dimensional Heisenberg algebra with basis $e_n$, $h$, $n\in\Z\setminus0$ and relations
        $$
        [e_n,e_m]=n\delta_{n,-m}h,
        $$
        where we set $e_n:=f_{-n}$ for $n<0$.
        The element $e_n$ has degree $n\alpha$.
        \item Let $\h=0$ and $\h^\vee=\C\alpha$ with $\alpha_n=n\alpha$. Then we get the abelian Lie algebra with basis $e_n$, $n\in\Z$.
        \item Let $\h=\C h$ and $\h^\vee=0$. The construction gives us the Heisenberg algebra, but the condition (\ref{eq:finiteness_condition}) is not satisfied and we have an infinite-dimensional zero root space.
    \end{enumerate}
\end{example}

% \subsection{Kac-Moody algebras associated with lattices}

% \begin{example}
%     Let $L$ be a Lorentzian lattice, i.e. a finite free abelian group with a symmetric non-degenerate bilinear form of signature $(n-1,1)$.  Choose elements $\alpha_i\in L$ for $i\in I$, and assume that
%     $$
%     \left(\sum_i\Z_{\ge 0}\alpha_i\right)\cap \left(\sum_i\Z_{\le 0}\alpha_i\right)=\{0\}.
%     $$
%     This condition means that the set of roots decomposes into positive and negative roots. We define $\h:=L_{\C}^*$, identify $\h$ with $\h^*$ via the bilinear form from $L$, and set $h_i:=\alpha_i$. Then $A$ is 
% \end{example}

% \begin{example}
%     Let $L$ be a lattice, i.e. a finite free abelian group with a symmetric non-degenerate bilinear form. Choose elements $\alpha_i\in L$ for $i\in I$, and assume that
%     $$
%     \left(\sum_i\Z_{\ge 0}\alpha_i\right)\cap \left(\sum_i\Z_{\le 0}\alpha_i\right)=\{0\}.
%     $$
%     This condition means that the set of roots decomposes into positive and negative roots. We define $\h:=L_{\C}^*$, identify $\h$ with $\h^*$ via the bilinear form from $L$, and set $h_i:=\alpha_i$. Then $A$ is 
% \end{example}

%By the above discussion, we can think of $\g(\mc F)$ as the subalgebra generated by $\sl_2$ and Heisenberg triples, with the assumption that the positive part of each such triple commutes with the negative part of each another one, quotiented by the condition that representation of each triple generated by a positive element of another triple is always irreducible.

\section{VBKM algebras}\label{s:VBKM_algebras}
\subsection{The Virasoro algebra}
We start with recalling the definition of the Virasoro algebra. First of all, we define the {\bfseries Witt algebra} to be the Lie algebra $\operatorname{Der}\C[t,t^{-1}]$ of derivations of $\C[t,t^{-1}]$. It has a basis $L_n':=t^{n+1}\partial_t$, and these elements satisfy the following relations:
$$
[L_n',L_m']=(n-m)L_{n+m}',\qquad n,m\in\Z.
$$

The {\bfseries Virasoro algebra} $\Vir$ is a non-trivial central extension of the Witt algebra. It turns out that any two nontrivial central extensions are isomorphic as Lie algebras. The most common form to write such an extension is taking a vector space with basis $L_n$ and $c$, $n\in\Z$, with the Lie bracket defined by
$$
[L_n,L_m]=(n-m)L_{n+m}+\delta_{n,-m}\frac{n^3-n}{12}c.
$$
Note that if we replace $L_0$ with $\tilde L_0:=L_0+\frac {a+1}{24}c$ and set $\tilde L_n:=L_n$, then this relation will become
$$
[\tilde L_n,\tilde L_m]=(n-m)\tilde L_{n+m}+\delta_{n,-m}\frac{n^3+an}{12}c.
$$
So, the linear term can be chosen arbitrarily. However, the standard choice adds an interesting new feature: the triple $L_{1},L_0,L_{-1}$ is an $\sl_2$-triple.

We describe the structure formula for the Virasoro Verma modules of central charge $c$. According to \cite[Theorem 8.1]{KR97}, a highest weight vector in the Verma module $M(c,h)$ has weight $h+N$, where $N=rs$ and $h=h_{r,s}$ for some positive integers $r,s$ and
\begin{equation}\label{eq:Virasoro_highest_weights}
    h_{r,s}(c) = \frac{1}{48} \left[ (13 - c)(r^2 + s^2) + \sqrt{(c-1)(c-25)} (r^2 - s^2) - 24rs - 2 + 2c \right].
\end{equation}

% If $(1-c)(25-c)$ is not a square, the equality $h=h_{r,s}$ is only possible when $r=s$ since $h$ is an integer. In this case,
% $$
% h_{r,r}=\frac1{24} (1-c)(r^2-1).
% $$

\subsection{The algebra $\widetilde{\g}(\mc F)$}\label{s:tilde_g}
Our goal is to allow adding Virasoro algebras in the construction of BKM algebras, along with $\sl_2$ and Heisenberg algebras. Any such Virasoro algebra will have one additional root and two additional coroots: one root and coroot coming from the central charge, and one coroot coming from $L_0$. The next definition gives the conditions which those additional roots and coroots must satisfy.
\begin{definition}\label{def:enh_gen_Cartan_datum}
    Fix sets $I$ and $J$, possibly empty. An {\bfseries enhanced generalized Cartan datum} is a tuple $(\h,\h^\vee,I,J,\alpha_i,h_i,\gamma_j,c_j,b_j)$ such that
    \begin{itemize}
        \item $(\h,\h^\vee,I,\alpha_i,h_i)$ is a generalized Cartan datum.
        \item $b_j,c_j\in \h$ and $\gamma_j\in\h^\vee$ with $\gamma_i(c_i)=0$ and $\gamma_i(b_i)=-1$.
        \item If $i\in I_{>0}$, then $\gamma_j(h_i)\in\Z_{\le 0}$.
        %\item $\gamma_j(h_i),\beta_j(h_i)\le 0$
        \item We assume the finiteness condition (\ref{eq:finiteness_condition}) for the joint multiset of {\bfseries simple roots} $\alpha_i,n\gamma_j$, $n\in\Z_{\ge 0}$.
        \item $\alpha_i(c_j)=0$ if and only if $\gamma_j(h_i)=0$.
        \item $\gamma_i(c_j)=0$ if and only if $\gamma_j(c_i)=0$.
        \item $\alpha(c_j)=0$ implies $\alpha(b_j)=0$ for any $\alpha\in\{\alpha_i,\gamma_k:i\in I,k\in J\setminus j\}$.
    \end{itemize}
\end{definition}
% \begin{remark}
%     Note that we impose no conditions on the elements $\beta_j$, which will serve as the elements $L_0$ in Virasoro algebras. The reason is that the root corresponding to the $j$-th Virasoro algebra will be $c_j$, its central charge.
% \end{remark}
\begin{remark}
    For how the last condition is used, see part (ii) in the proof of Proposition \ref{p:ideal_in_g_tilde}, Computation \ref{comp:why_not_submodules_of_Vermas}, and discussions around (\ref{eq:special_Virasoro_condition}) and (\ref{eq:generic_Virasoro_condition}).
\end{remark}

Now, it is time to define VBKM algebras. Let $\mc F$ be an enhanced generalized Cartan datum, and let $\mc F':=(\h,\h^\vee,I,\alpha_i,h_i)$ be its generalized Cartan subdatum. We define the Lie algebra $\tilde\g(\mc F)$ generated by $\tilde\g(\mc F')$, $L_n^{j}$, $c_j$ with the following additional relations:
\begin{align*}
[L_n^j,L_m^j]&=(n-m)L_{n+m}^j+\delta_{n,-m}(n^3-n)c_j,\\
[\h,\h]&=0,\\
[h_i,L^j_n]&=n\gamma_j(h_i)L^j_n,\\
L_0^j&=b_j,\\
[L^j_n,L^k_m]&=0,\qquad nm<0, k\ne j,\\
[e_i,L^j_n]&=[f_i,L^j_{-n}]=0,\qquad n<0.
\end{align*}

To prove a statement analogous to Theorem \ref{t:tilde_g_facts}, we need a little preparation. We will use a method similar to the one used in the proofs of \cite[Proposition 1.1]{Jur96} and \cite[Theorem 1.2]{Kac90}. For $i\in I$, let $\g_i$ be the algebra generated by $e_i,h_i,f_i$ with the imposed relations (hence isomorphic to $\sl_2$ or Heisenberg), and for $j\in J$, let $\Vir^j$ be the Virasoro algebra generated by $L_n^j$ and $c_j$. Note that the central charge of $\Vir^j$ equals $12c_j$. We have usual triangular decompositions
    \begin{align*}
        \g_i&=\n_{i,-}\oplus\h_i\oplus\n_{i,+},\\
        \Vir^j&=\Vir^j_{<0}\oplus \lge L_0^j,c_j\rge\oplus \Vir_{>0}^j.
    \end{align*}
    
    Let $G$ be the free product of all $\g_i$ and $\Vir^j$ and $G_-$ be the free product of all $\n_{i,-}$ and $\Vir^j_{<0}$. Choose $\lambda\in \h^*$, and let 
    $$
    \widetilde M(\lambda):=U(G_-),
    $$
    which also equals the free product of $U(\n_{i,-})$ and $U(\Vir^j_{<0})$. We will construct a representation of $G$ on $\widetilde M(\lambda)$ depending on $\lambda$ and show that it descends to a representation of $\tilde\g(\mc F)$.

    It is enough to specify how generators act. We let the generators of $G_-$ act by left multiplication. The elements of $\h$ act semisimply by
    \begin{align*}
        h.1&=\lambda(h),\\
        [h,f_i]&=-\alpha_i(h)f_i,\\
        [h,L_n^j]&=n\gamma_j(h)L_n^j,\qquad n<0.
    \end{align*}
    
    For the generators of $G_+$ and $\h$, we define them to act by derivations on $\widetilde M(\lambda)$ as an algebra. For this, it is enough to specify where they send the generators and show that they preserve the relations in $G_-$. We set:
    \begin{align*}
        e_i.f_j&=\delta_{i,j}h_i.1=\delta_{i,j}\lambda(h_i),\\
        e_i.L_n^j&=0,\\
        L_m^j.f_i&=0,\\
        L_m^j.L_n^{j_0}&=\delta_{j,j_0}\left[(m-n)L_{m+n}^j+\delta_{m,-n}(m^3-m)c_j\right].1,
    \end{align*}
    where the value in the last line is identified by induction on $m\le 0$ using $m+n>m$.

\begin{lemma}\label{l:tilde_g_Virasoro_representation}
    The just defined operators on $\widetilde M(\lambda)$ form a representation of $\tilde \g(\mc F)$.
\end{lemma}
\begin{proof}
    First, we prove that the operators are well-defined, i.e. they preserve all the relations in $G_-$. The only relations in $G_-$ come from Virasoro relations. So, for $m,n>0$ we check:
    \begin{align*}
        [h,L_n^jL_m^{j}-L_m^{j}L_n^j-(n-m)L_{m+n}]&=(m+n)(L_n^jL_m^{j}-L_m^{j}L_n^j-(n-m)L_{m+n}),\\
        [e_i,L_n^jL_m^{j}-L_m^{j}L_n^j-(n-m)L_{m+n}]&=0,\\
    \end{align*}
    and all left multiplications by elements in $G_-$  preserve the ideal of relations by the very definition of an ideal. It remains to check the action of positive Virasoro operators: for $k\ge0$ and $m,n\le 0$, we compute
    \begin{align*}
        [L_k^{j_0},(n-m)L_{m+n}]&=(n-m)\delta_{j,j_0}\left[(k-m-n)L_{m+n+k}^j+\delta_{k,-m-n}(k^3-k)c_j\right]
    \end{align*}
    and
    \begin{align*}
        [L_k^{j_0},L_n^jL_m^{j}-L_m^{j}L_n^j]&=[[L_k^{j_0},L_n^j],L_m^j]+[L_n^{j},[L_k^{j_0},L_m^j]]\\
        &=\delta_{j,j_0}(k-n)[L_{k+n}^j,L_m^j]+\delta_{j,j_0}(k-m)[L_{n}^j,L_{k+m}^j]\\
        &=(n-m)\delta_{j,j_0}\left[(m+n+k)L_{m+n-k}^j+\delta_{k,-m-n}(k^3-k)c_j\right],
    \end{align*}
    where in the second equality we used that $\gamma_j(c_j)=0$, hence commutativity of $c_j$ with $L_m^j$ and $L_n^j$, and the third equality follows from induction on $k<0$ with the usage of $k+n>k$ and $k+m>k$. The two expressions are equal, which proves well-definiteness.
    
    Now, we prove that the defined operators satisfy the relations in $\tilde \g(\mc F)$. Recall that for any associative algebra $A$ and its derivation $D$, we have
    $$
    [D,a]=D(a),\qquad \text{for all }a\in A,
    $$
    where we regarded $a$ and $D(a)$ as left multiplications by these elements. Applying it to our situation, we get
    \begin{align*}
        [e_i,f_i].(-)&=(e_i.f_i).(-)=\delta_{ij}h_i.(-),\\
        [e_i,L_n^j].(-)&=(e_i.L_n^j).(-)=0,\qquad n<0,\\
        [L_m^j,f_i].(-)&=(L_m^j.f_i).(-)=0,\qquad m>0,\\
        [L_m^j,L_n^{j_0}].(-)&=(L_m^j.L_n^{j_0}).(-)=\delta_{j,j_0}\left[(m-n)L_{m+n}^j+\delta_{m,-n}(m^3-m)c_j\right].(-).
    \end{align*}
    From this and earlier definition of the action of $\h$, we get all the relations of $\tilde\g(\mc F)$ except the Virasoro relations between $[L_m^j,L_n^{j_0}]$ for $n,m\ge 0$ or $n,m<0$.

    We check these relations separately. The case $n,m<0$ is evident since the operators $L_m^j,L_n^{j_0}$ are left multiplications in the universal enveloping algebra of the Lie algebras where the desired relations are satisfied.

    Let us deal with the case $m,n>0$. Since $[L_m^j,L_n^{j_0}]-\delta_{j,j_0}(m-n)L_{m+n}^j$ is a derivation, it is enough to check that it vanishes on generators. Vanishing on $f_i$ is evident by commutativity of positive Virasoro operators with $f_i$. Thus, it remains to check that
    $$
    ([L_m^j,L_n^{j_0}]-\delta_{j,j_0}(m-n)L_{m+n}^j).L_k^{j_1}=0
    $$
    for $k<0$. Note that it is zero unless $j=j_0=j_1$, so let us assume this. We run induction on $m+n\le 0$. The base case $m=n=0$ is clear. We have
    \begin{align*}
        [L_m^j,L_n^{j}].L_k^{j}&=[[L_m^j,L_n^{j}],L_k^{j}].1\\
        &=[[L_m^j,L_k^j],L_n^j].1+[L_m^j,[L_n^j,L_k^{j}]].1\\
        &=(m-k)[L_{m+k}^j,L_n^j].1+(n-k)[L_m^j,L_{n+k}^j].1\\
        &=(m-n)\left[(m+n-k)L_{m+n+k}-\delta_{m+n+k,0}(k^3-k)c_j\right].1.
    \end{align*}
    where in the third line we used that $\gamma_j(c_j)=0$, whence the commutativity with the central charge, and the last line follows from the induction hypothesis and the known case $m\ge0, n<0$. We get the same result with $(m-n)L_{m+n}^j.L_k^{j_1}$, so we are done.
\end{proof}
\begin{theorem} \label{t:tilde_g_Virasoro_facts}
Let $\mc F$ be an enhanced generalized Cartan datum. Then
    \begin{enumerate}[label=(\roman*)]
        \item The Lie algebra $\tilde \g(\mc F)$ is non-zero.
        \item Let $\n_+$ ($\n_-$) be the subalgebra of $\tilde \g(A)$ generated by $e_i$ and $L_n^j$ ($f_i$ and $L_{-n}^j$) for $i\in I$ and $n>0$. We have a direct sum decomposition
        $$
        \tilde \g(\mc F)=\tilde\n_-\oplus\h\oplus\tilde\n_+.
        $$
        Moreover, $\h$ acts semisimply on $\tilde \g(\mc F)$ and $\n_+$ is freely generated by $e_i$ and 
        $\Vir_{>0}^j:=\lge L_n^j:n>0\rge$.
        \item There is an involution $\omega$ on $\tilde\g(\mc F)$ with
        \begin{align*}
            \omega(e_i)&=-f_i,\\
            \omega(h_i)&=-e_i,\\
            \omega(h)&=-h,\\
            \omega(L_n^j)&=-L_{-n}^j.
        \end{align*}
        \item There exists the largest ideal $\r$ in $\tilde \g(A)$ which does not intersect $\h$. We have 
        $$
        \r=\r\cap \tilde\n_+\oplus \r\cap\tilde\n_-.
        $$
    \end{enumerate}
\end{theorem}
\begin{proof}
    We start with $(i)-(iii)$ using the representation $\widetilde M(\lambda)$ of $\tilde \g(\mc F)$ studied in Lemma \ref{l:tilde_g_Virasoro_representation}. The operators from $\h$ act nontrivially, which implies that the representation is nontrivial, hence non-zero. Moreover, by construction, the map
    $$
    U(\tilde \n_-)\to \widetilde M(0)
    $$
    sending $x$ to $x.1$ is a bijection. The quotient map $\widetilde M(0)\to U(\tilde\n_-)$ is inverse to it, which proves that $U(\tilde\n_-)\simeq U(G_+)$, showing that $\tilde \n_-$ is freely generated by $f_i$ and $\Vir^j_{<0}$.

    It is a straightforward check that the map defined in $(iii)$ satisfies all the relations of $\tilde\g(\mc F)$, and hence gives an involution on the algebra. This involution swaps $\tilde\n_+$ and $\tilde\n_-$, which makes the just proven statements true for $\tilde\n_+$ as well.

    Using the relations between generators of $\tilde\n_+$ and $\tilde \n_-$ like in the standard case of Kac-Moody Lie algebras, we can see that $\g=\tilde\n_-+\h+\tilde\n_+$. This sum is direct since the $\h$-weights of $\n_+$ lie in the non-zero part of $\Z_{\ge 0}\lge \alpha_i,\gamma_j,\beta_j\rge$, while the weights of $\tilde\n_+$ lie in the non-zero part of $\Z_{\le 0}\lge \alpha_i,\gamma_j,\beta_j\rge$.

    For $(iv)$, note that because of the semisimplicity of the $\h$-action on $\tilde \g(\mc F)$, any ideal is a semisimple representation of $\h$ as well. This shows the last decomposition and that if two ideals do not intersect $\h$, then their sum does not, proving existence of the largest such ideal.
\end{proof}

\subsection{Ideals in $\tilde{\g}(\mc F)$}\label{s:ideals_in_tilde_g}
For the rest of this section, we fix an enhanced generalized Cartan datum $\mc F$ and suppress it from our notation (so that $\tilde\g:=\tilde\g(\mc F)$). We turn to studying the largest ideal of $\tilde\g$ and elements it contains. Consider four types of representations:
\begin{enumerate}[label=(\roman*)]
    \item $\g_i$ acting on $U(\g_i)f_k$ for $k\ne i$.
    \item $\g_i$ acting on $U(\g_i)L_n^j$ for $n<0$.
    \item $\Vir^j$ acting on $U(\Vir^j)f_i$ with $\alpha_i(c_j)=\alpha_i(b_j)=0$. 
    \item $\Vir^j$ acting on $U(\Vir^j)L_n^k$ for $n<0$, $k\ne j$, with $\gamma_k(c_j)=\gamma_k(b_j)=0$.
\end{enumerate}
All these modules are isomorphic to Verma modules of the corresponding Lie algebra: $M(-\alpha_k(h_i))$, $M(n\gamma_j(h_i))$, $M(0,0)$, and $M(0,0)$, respectively. By the standard structure theory of Verma modules, each of these modules possesses the largest proper submodule. Let $X_-$ be the set containing
\begin{itemize}
    \item all highest weight vectors generating all these proper submodules in $\tilde\n_-$,
    \item All $L_m^j.f_i$ and $L_m^j.L_n^k$ for $m<0$ in cases $(iii)-(iv)$.
\end{itemize}
Define $X_+\subset \tilde \n_+$ in a similar way (so that $X_+=\omega(X_-)$).

\begin{proposition}\label{p:ideal_in_g_tilde}
    The ideal $\tilde I_-$ ($\tilde I_+$) of $\tilde \g$ generated by $X_-$ $(X_+)$ is contained in $\tilde\n_-$ ($\tilde\n_+$). In particular, $\tilde I:=\tilde I_-\oplus \tilde I_+$ is an ideal of $\tilde \g$ not intersecting $\tilde\h$.
\end{proposition}
\begin{proof}
    Choose $x\in X_-$. We will study the action of $\tilde\n_+$ on $x$. We need to consider $4$ cases $(i)-(iv)$ as above for the representation the element $x$ lies in:
    \begin{enumerate}[label=(\roman*)]
        \item It is enough to show that $\tilde\n_+x=0$, for which it is enough to show that the generators of $\n_+$ annihilate $x$. Since $x$ is generated by $f_i$ and $f_k$, it is annihilated by all generators of $\tilde n_+$ other than $e_i$ and $e_k$ since they commute with $f_i$ and $f_k$. Since it is a highest weight vector for $\g_i$, it is annihilated by $e_i$. It remains to check annihilation by $e_k$.
        
        We have two cases here: $\alpha_k(h_i)=0$ or $\alpha_k(h_i)\ne0$. In the first case, $x=[f_i,f_k]$ and
        $$
        [e_k,[f_i,f_k]]=[f_i,h_k]=\alpha_i(h_k)=0,
        $$
        where we used our assumption that $\alpha_i(h_k)=0$ iff $\alpha_k(h_i)=0$.
        
        In the second case, the highest weight is non-zero, which implies that if $\g_j$ was the Heisenberg algebra, the Verma module would be irreducible (see \cite[Proposition 1.7.2]{FLM89}). Therefore, $\g_j=\sl_2$. By representation theory of $\sl_2$, we can write $x=(\operatorname{ad}f_i)^{1-\frac{2a_{ij}}{a_{ii}}}.f_k$. Since $e_k$ commutes with $f_i$, we have
        $$
        e_k.(\operatorname{ad}f_i)^{1-\frac{2a_{ij}}{a_{ii}}}.f_k=(\operatorname{ad}f_i)^{1-\frac{2a_{ij}}{a_{ii}}}.[e_k,f_k]=(\operatorname{ad}f_i)^{1-\frac{2a_{ij}}{a_{ii}}}.h_k.
        $$
        By our assumption, $a_{ik}=\alpha_k(h_i)\ne 0$, which implies that $1-\frac{2a_{ij}}{a_{ii}}\ge 2$. Then
        $$
        (\operatorname{ad}f_i)^{1-\frac{2a_{ij}}{a_{ii}}}.h_k=(\operatorname{ad}f_i)^{-1-\frac{2a_{ij}}{a_{ii}}}[f_i,[f_i,h_k]]=0,
        $$
        as $[f_i,h_k]$ is proportional to $f_i$. This finishes the proof in this case.
        \item Let $X_i^j$ be the set of highest weight vectors in this case. It is enough to show that $\tilde\n_+X_i^j\subset U(f_i)X_i^j$. Using similar arguments as in $(ii)$, we conclude that the generators other than $L_m^j$ for $m\ge 0$ annihilate $X_i^j$, so it is enough to prove that $L_n^jX_i^j\subset U(f_i)X_i^j$ for $n>0$.

        We have two cases: $\gamma_j(h_i)=0$ or $\gamma_j(h_i)\ne 0$. In the first case, recall that out assumptions on the extended Cartan datum imply that $\beta_j(h_i)=0$. Then $h_i$ commutes with $\Vir^j$, which implies that $X_i^j$ is spanned by elements of the form $x=[f_i,L_n^j]$. Using similar arguments as in $(ii)$, we conclude that the generators other than $L_m^j$ for $m\ge 0$ annihilate $x$. We have
        $$
        L_m^j.[f_i,L_n^j]=(m-n)[f_i,L_{m+n}^j],
        $$
        which is $0$ if $m+n\ge 0$ and lies in $U(f_i)X_i^j$ otherwise. So, this case is done.

        Now, assume that $\gamma_j(h_i)\ne 0$. Then, like in $(i)$, we get that $\g_i=\sl_2$ and $X_i^j$ is spanned by elements of the form 
        $$
        x=(\operatorname{ad}f_i)^{1+\frac{2n\gamma_j(h_i)}{a_{ii}}}.L_n^j.
        $$
        Then
        \begin{align*}
            L_m^j.x&=(\operatorname{ad}f_i)^{1+\frac{2n\gamma_j(h_i)}{a_{ii}}}.[L_m^j,L_n^j]\\
            &=(m-n)(\operatorname{ad}f_i)^{1+\frac{2n\gamma_j(h_i)}{a_{ii}}}.L_{m+n}^j+\delta_{m,-n}(m^3-m)(\operatorname{ad}f_i)^{1+\frac{2n\gamma_j(h_i)}{a_{ii}}}c_j
        \end{align*}
        Since $m+n>n$ and $\gamma_j(h_j)/a_{ii}<0$ (otherwise the Verma module would be irreducible), we have
        $$
        1+\frac{2n\gamma_j(h_i)}{a_{ii}}>1+\frac{2(n+m)\gamma_j(h_i)}{a_{ii}},
        $$
        showing that the first summand lies in $U(f_i)X_{i}^j$. The second summand is zero by the same argument as in the end of case $(i)$.
        \item This relation imposes commutativity $[\Vir^j_{< 0},f_i]=0$. By our assumption on the Cartan datum, we also have $c_j(h_i)=b_j(h_i)=0$, hence this relation has already been considered in $(ii)$ above.
        \item This relation imposes commutativities $[L_m^j,L_n^k]=0$. As before, it is enough to check the action of $L_s^k$ for $s\ge0$:
        $$
        L_s^k.[L_m^j,L_n^k]=(s-n)[L_m^j,L_{n+s}^k]+\delta_{n,-s}(s^3-s)[L_m^j,c_k]=0,
        $$
        which shows that the ideal generated by this relation lies in $\tilde\n_+$. This finishes the proof.
    \end{enumerate}
\end{proof}

We define $\g=\g(\mc F):=\tilde \g/\tilde I$ and call it the {\bfseries Virasoro-Borcherds-Kac-Moody algebra} associated to $\mc F$. 

Let us study this algebra in more detail. Note that for the action of $\g_i$, we added the largest submodules of the Verma modules into the ideal $I$, while we only did it with the actions of $\Vir^j$ if the corresponding Virasoro module was $M(0,0)$. A natural question to ask is why we did not add highest weight vectors of other $M(c,h)$ into $I$.

\begin{computation}\label{comp:why_not_submodules_of_Vermas}
    In short: for most such relations, if we add one of them, the ideal will intersect $\tilde\h$.
    
    Let us start with looking at the ideal $J$ generated by some Virasoro highest weight vector of the form $P.f_i$ for some $P\in U(\Vir^j_{<0})$. Then we have
    $$
    [e_i,P.f_i]=P.h_i=-\deg P\cdot \gamma_j(h_i) P.
    $$
    If $\gamma_j(h_i)\ne 0$, then produces a non-zero element of $U(\Vir^j_{<0})$. Applying elements of $\Vir^j_{>0}$ to it, we can reach a non-zero linear combination of $b_j$ and $c_j$, hence the ideal would intersect $\h$.

    Now, we treat $U(\Vir^j_{<0})L_n^k$ for $n<0$. Let $PL_n^k$ be a nontrivial highest weight vector in it with $P\in U(\Vir^j_{<0})$. We have
    \begin{equation}
    [L_{-n}^k,PL_n^k]=P.\left(-2nb_k-(n^3-n)c_k\right)=\deg P\cdot\gamma_j(-2nb_k-(n^3-n)c_k)P,
    \end{equation}
    which by the same arguments as in the previous case implies that
    $$
    \gamma_j(-2nb_k-(n^3-n)c_k)=0.
    $$
    Suppose that $n<-1$. Then 
    $$
    [L_{1}^k,PL_n^k]=(1-n)PL_{n+1}^k,
    $$
    so the same computation gives
    $$
    \gamma_j(-2(n+1)b_k-((n+1)^3-(n+1))c_k)=0.
    $$
    The two equations that we got are linearly independent, hence imply that $\gamma_j(b_k)=\gamma_j(c_k)=0$, i.e. the commutative situation.

    Now, assume that $n=-1$. Then we get $\gamma_j(b_k)=0$. By the definition of an enhanced generalized Cartan datum, it gives $\gamma_k(b_j)=0$. Therefore,
    $$
    U(\Vir^j_{<0})L_{-1}^j\simeq M(\gamma_k(c_j),0),
    $$
    so the relation with $P=L_{-1}^j$ implies all other relations.
\end{computation}

Based on this, we state the following Proposition.
\begin{proposition}\label{p:new_ideal_in_special_VBKM}
    The ideal $\tilde J$ of $\g$ generated by relations 
\begin{itemize}
    \item[(iv')] $[L_{-1}^j,L_{-1}^k]$ and $[L_1^j,L_1^k]$ if  $\gamma_j(b_k)=\gamma_k(b_j)=0$
\end{itemize}
    does not intersect $\h$.
\end{proposition}
\begin{proof}
    We check how $\Vir^j_{> 0}$ and $\Vir^k_{>0}$ act on the first relation:
    \begin{align*}
        [L_n^j,[L_{-1}^j,L_{-1}^k]]=(n+1)[L_{n-1}^j,L_{-1}^k]=\delta_{n,1}(n+1)\gamma_k(b_j)=0,\\
        [L_n^k,[L_{-1}^j,L_{-1}^k]]=(n+1)[L_{-1}^j,L_{n-1}^k]=-\delta_{n,1}(n+1)\gamma_j(b_k)=0.
    \end{align*}
    Thus, the ideal $J_-$ generated by the first relation in contained in $\n_-$. We get a similar result for the second relation, so we are done.
\end{proof}

See Example \ref{e:special_VBKM_realization}.
\begin{remark}
    This additional ideal $\tilde J$ makes it reasonable to define a VBKM algebra as $\g/\tilde J$. However, we keep our original definition of a VBKM algebra because of its resemblance with BKM algebras (and direct connection to them established in the following sections).
    
    It is quite likely that $\g/\tilde J$ is a more natural definition of a VBKM algebra from other points of view, e.g. because it is the algebra obtained by divided our by all Serre relations: old ones and their Virasoro analogues. As we establish in future sections (see conditions (\ref{eq:special_Virasoro_condition}) and (\ref{eq:generic_Virasoro_condition}) and discussion around them), such algebra is of a special nature as opposed to generic. This can be a subject of future research.
\end{remark}

\begin{definition}
    The {\bfseries Weyl group} $W$ of a VBKM algebra $\g$ is the group generated by simple reflections corresponding to elements in $I_{>0}$.
\end{definition}
Thus, the Weyl group of a VBKM algebra equals the Weyl group of its Kac-Moody subalgebra.
\begin{proposition}\,\label{p:Weyl_group_on_g}
    \begin{enumerate}[label=(\roman*)]
        \item We have the Cartan decomposition
        $$
        \g=\n_-\oplus\h\oplus\n_+
        $$
        for $\n_-:=\tilde\n_-/\tilde I_-$ and $\n_+:=\tilde\n_+/\tilde I_+$.
        \item Define operators
        $$
        n_i:=\exp\operatorname{ad}e_i\cdot \exp\operatorname{ad}(-f_i)\exp\operatorname{ad}e_i,\qquad i\in I_{>0}.
        $$
        They act by automorphisms of $\g$ and its modules in the category $\mc O$, and send the weight/root space $\alpha$ to the weight/root space $s_i(\alpha)$, where $s_i$ is the simple reflection corresponding to $i$.
        \item For any $\alpha$ in the root lattice and $w\in W$, we have
        $$
        \dim\g_{\alpha}=\dim\g_{w(\alpha)}.
        $$
    \end{enumerate}
\end{proposition}
\begin{proof}
    The first item is clear. The operators $n_i$ are well-defined on $\g$ because the actions of $e_i,f_i$-s corresponding to indices in $I_{>0}$ are locally nilpotent, which is the case since we have quotiented out by Serre relations. For the rest of $(ii)$, see \cite[Lemma 3.8]{Kac90}. Item $(iii)$ follows from $(ii)$.
\end{proof}

\section{Connection to BKM algebras}\label{s:BKM-VBKM_connections}
\subsection{VBKM algebras as deformations of BKM algebras}\label{s:ultimate_algebra}
To state our main structural results, we need to introduce a bit more notation. Let $\g_{\Vir}$ be the algebra $\g$ we are studying. Let $\g_{\Heis}$ be the BKM algebra associated to the generalized Cartan datum 
$$
(\h,\h^\vee,I\sqcup \bigsqcup\nolimits_{n\in\Z_{> 0}}nJ,\alpha_i,n\gamma_j,h_i,nc_j:n\in\Z_{>0}).
$$
Intuitively, we replace all the Virasoro algebras in the generating set of $\tilde\g$ by Heisenberg algebras. Note that
\begin{itemize}
    \item $\ch\Vir^j=\ch\Heis^j$.
    \item Pairwise relations between $\g_i, \Vir^j$ generating the ideal $I$ are identical to pairwise relations between $\g_i, \Heis^j$ in the BKM algebra $\g_{\Heis}$.
    \item The elements $b_j$ remain in $\h$ and act by scalars on root spaces of the associated BKM algebra. We could think of $\widehat\Heis^j:=\Heis^j\rtimes \C b_j$ being building pieces of the BKM algebra instead of just $\Heis^j$, where $[b_j,x_n^j]=-nx_n^j$ for the standard basis $x_n^j$ of $\Heis^j$.
\end{itemize}

\begin{definition}
    We call an enhanced generalized Cartan datum $\mc F=(\h,\h^\vee,I,J,\alpha_i,h_i,\gamma_j,c_j,b_j)$ {\bfseries symmetrizable} if there exists an automorphism of $\h$ acting on elements $h_i,c_j,b_j$ by rescaling (by possibly distinct constants) such that
    \begin{itemize}
        \item After rescaling, the generalized Cartan matrix of the datum $(\h,\h^\vee,I,J,\alpha_i,h_i)$ becomes symmetric.
        \item After rescaling, $\alpha_i(c_j)=\gamma_j(h_i)$ and $\gamma_k(c_j)=\gamma_j(c_k)$.%, and $\gamma_k(b_j)=\gamma_j(b_k)$.
    \end{itemize}
\end{definition}

% We state the following Conjecture.
% \begin{conjecture}
%     If $\mc F$ is symmetrizable, then
%     $$
%     \ch \g_{\Vir}=\ch\g_{\Heis}.
%     $$
%     In particular, $\g_{\Vir}$ satisfies the BKM denominator identity.
% \end{conjecture}

We would like to prove a statement analogous to Theorem \ref{t:tilde_g_facts}. To do it, we develop a deformation formalism that would connect $\g_\Vir$ and $\g_\Heis$ together. Let $\n_\Vir$ and $\n_\Heis$ be the positive parts of these two algebras. Recall that we have presentations
\begin{align*}
    \n_\Vir&=\frac{*_{i\in I}\g_i*_{j\in J}\Vir^j_{>0}}{\lge (\operatorname{ad} e_i)^{1-2a_{ik}/a_{ii}}e_k,(\operatorname{ad} e_i)^{1+2n\gamma_j(h_i)/a_{ii}}L_n^j:i\in I_{>0}\rge},\\
    \n_\Heis&=\frac{*_{i\in I}\g_i*_{j\in J}\Heis^j_{>0}}{\lge (\operatorname{ad} e_i)^{1-2a_{ik}/a_{ii}}e_k,(\operatorname{ad} e_i)^{1+2n\gamma_j(h_i)/a_{ii}}h_n^j:i\in I_{>0}\rge}.
\end{align*}
Now, consider the ultimate algebra
$$
\n_{\Vir,\Heis}:=\frac{\n_\Vir*\n_\Heis}{\lge[L_n^j,h_m^j]=-mh_{n+m}^j:j\in J\rge}.
$$
In fact, this algebra is the positive part of the ultimate algebra $\g_{\Vir,\Heis}$, which is constructed from $\tilde\g_{\Vir,\Heis}$ in the same way as in Section \ref{s:ideals_in_tilde_g}, while $\tilde\g_{\Vir,\Heis}$ is given by the analogous construction and the arguments from Section \ref{s:tilde_g}, where $\Vir^j$ is replaced with $\Vir^j\ltimes\Heis^j$. Since the proofs are similar to those of Lemma \ref{l:tilde_g_Virasoro_representation} and just contain more cases to check, we refer to the Appendix for them.

\subsection{A family of Virasoro-Heisenberg algebras over $\PP^1$}\label{s:families_of_Vir-Heis_algebras}
So, we turn to studying the algebra $\mathfrak{v}:=\Vir\ltimes\Heis$ for a while. It has basis $L_n$ for $n\in\Z$, $x_n$ for $n\in\Z\setminus 0$, and $c$, and the relations between $L_n$ and $h_m$ are
$$
[L_n,x_m]=-mx_{n+m}.
$$

For $t\in\C$ and $n\in\Z$, we define
$$
Y_n(s,t):=sL_n+tx_n\in\vv.
$$

\begin{lemma}\label{l:commutator_of_Heis-Vir_families}
    For $n,m\in\Z$, we have
    \begin{align*}
        [Y_n(s,t),Y_m(s,t)]&=(n-m)sY_{m+n}(s,t)+\delta_{m,-n}\left(s^2(n^3-n)+t^2n\right)c.
    \end{align*}
    In particular, $Y_n(1,t)+\delta_{n,0}\frac{t^2}2c$ and $c$ is the standard basis of the Virasoro algebra and $Y_n(0,1)$, $L_0$ is the standard basis of $\Heis\rtimes\C L_0$.
\end{lemma}
\begin{proof}
    This is a standard computation.
\end{proof}

% We turn to studying the algebra $\Vir\ltimes\Heis$ for a while. It has basis $L_n$ for $n\in\Z$, $x_n$ for $n\in\Z\setminus 0$, and $c$, and the relations between $L_n$ and $x_m$ are
% $$
% [L_n,x_m]=-mx_{n+m}.
% $$

% For $t\in\C$ and $n\in\Z$, we define
% $$
% Y_n(s,t):=sL_n+tx_n.
% % \begin{cases}
% %     sL_n^j+tx_n^j,&n\ne 0,\\
% %     s^2L_0^j+\frac{t^2}2c,&n\ne 0.
% % \end{cases}
% $$

Now, we turn this into a family of Lie algebras over $\PP^1$. Let $\mc V_n$ be the subbundle of $\mathfrak{v}\otimes \mc O_{\PP^1}$ such that over $[s:t]$, it is the line in $\g$ spanned by $Y_n(s,t)$ if $n\ne 0$ and entire $\g_0$ over $0$. We set $\mc V:=\bigoplus_{n\in\Z}\mc V_n$ and write $\mc V_0=\mc C\oplus\mc L$ where $\mc C$ is spanned by $c$ and $\mc L$ is spanned by $L_0$.

By Lemma \ref{l:commutator_of_Heis-Vir_families}, each fiber of $\mc V$ is closed under the Lie bracket coming from $\mathfrak{v}_0$. Moreover, the form of elements $Y_n(s,t)$ gives that $\mc G_n$ is isomorphic to the tautological bundle $\mc O_{\PP^1}(-1)$, while $\mc G_0$ is isomorphic to $\mc O_{\PP^1}^{\oplus 2}$. So, it is a Lie algebra family with the Lie bracket given by the following rules:
\begin{align*}
    \mc V_n\otimes \mc V_m=\mc O(-1)\otimes \mc O(-1)=\mc O(-2)&\xrightarrow{(n-m)s}\mc O(-1)=\mc V_{n+m},\qquad n+m\ne 0,nm\ne 0,\\
    \mc V_n\otimes \mc V_{-n}=\mc O(-1)\otimes \mc O(-1)=\mc O(-2)&\xrightarrow{((n^3-n)s^2+nt^2,\,2ns^2)}\mc O\oplus\mc O=\mc C\oplus\mc L=\mc V_{0},\\
    \mc L\otimes \mc V_m=\mc O\otimes \mc O(-1)=O(-1)&\xrightarrow{-m}\mc O(-1)=\mc V_m,\\
    \mc C\otimes \mc V_m&\to 0.
\end{align*}

By Lemma \ref{l:commutator_of_Heis-Vir_families}, the fiber $\mc V_{x}$ is isomorphic to $\Vir$ if $x\ne 0$ and to $\Heis$ if $x=0$.

We can construct this family from the fact that the Virasoro and Heisenberg algebras are central extensions of the Witt and polynomial algebras, respectively. Let $\mathfrak{w}$ be the Witt algebra and
$$
\mc W':=\mathfrak{w}\otimes \mc O=\bigoplus_{n\in\Z} \mc W_n,\qquad \mc W_n'=\mc O
$$ 
be the trivial family of Witt algebras over $\PP^1$. It has a central extension by $\mc O$ making it the trivial family of Virasoro algebras over $\PP^1$.

Now, twist this family by $\mc O(-1)$ and the Lie bracket by the embedding $\mc O(-1)\xrightarrow{s}\mc O$. In other words, we get the Lie algebra family $\mc W'(-1)$ with the Lie bracket given by
\begin{align*}
    \mc W_n'(-1)\otimes \mc W_m'(-1)=\mc O(-1)\otimes \mc O(-1)=\mc O(-2)&\xrightarrow{(n-m)s}\mc O(-1)=\mc W_{n+m}(-1).
\end{align*}
It is a non-trivial family over $\PP^1$ with fiber at $x\in\PP^1$ isomorphic to the Witt algebra when $x\ne 0$ and to the commutative polynomial Lie algebra $\C[t,t^{-1}]$ when $x=0$. This family admits a central extension $\mc W$ which equals $\mc G$ with $\mc L(-1)$ instead of $\mc L$. In other words, this central extension has $\mc O\oplus \mc O(-1)$ as the zero graded piece with
\begin{align*}
    \mc W_n\otimes \mc W_m=\mc O(-1)\otimes \mc O(-1)=\mc O(-2)&\xrightarrow{(n-m)s}\mc O(-1)=\mc W_{n+m},\qquad n+m\ne 0,nm\ne 0,\\
    \mc W_n\otimes \mc W_{-n}=\mc O(-1)\otimes \mc O(-1)=\mc O(-2)&\xrightarrow{((n^3-n)s^2+nt^2,\,2ns)}\mc O\oplus\mc O(-1)=\mc C\oplus\mc L(-1),\\
    \mc L(-1)\otimes \mc W_m=\mc O(-1)\otimes \mc O(-1)=O(-2)&\xrightarrow{-ns}\mc O(-1)=\mc W_m,\\
    \mc C\otimes \mc W_m&\to 0.
\end{align*}

\begin{remark}
    With these families defined, we can construct families of Verma modules over them. In particular, we get a family of Fock spaces turning into a family of vertex algebras -- the Virasoro vertex algebras over $\PP^1\setminus 0$ and the Heisenberg vertex algebra over $0$.

    Note that the family $\mc V$ over $\PP^1$ is the known one-parameter family of conformal structures on the Heisenberg vertex algebra. See \cite[Example 2.5.9]{FBZ04}
\end{remark}

\begin{remark}
    The algebras $\mc V$ and $\mc W$ differ by a twist of $\mc L$. This introduces a small difference in the study of Verma modules over these algebras, as the highest weight is taken in $\Gamma(X,\mc C\oplus\mc)$ in the case of $\mc V$ and in $\Gamma(X,\mc C\oplus\mc L(1))$ in the case of $\mc W$. In terms of parametrized families over $\PP^1$ the difference is that in the first case, the highest weight remains constant along $\PP^1$, while in the second case it varies as a linear function on $\PP^1$. In this paper, we will use $\mc V$ and will not be concerned with varying families of weights.
\end{remark}

\subsection{A family of VBKM algebras over $\mathbb P^1$}
We define a family of VBKM algebras over $\mathbb P^1$ interpolating between $\tilde\g_{\Vir}$ and $\tilde\g_{\Heis}$ inside the ultimate algebra $\tilde\g_{\Vir,\Heis}$ introduced in Section \ref{s:ultimate_algebra}. First of all, take the trivial family $\tilde\g_{\Vir,\Heis}\times \mathbb P^1$. We define a subfamily $\tilde {\mc G}\subset \tilde\g_{\Vir,\Heis}\times \mathbb P^1$ in the following way: the fiber $\tilde {\mc G}_x$ over a point $x=[s:t]$ is the Lie subalgebra of $\tilde \g_{\Vir,\Heis}$ generated by $e_i,f_i,\h,sL_n^j+tx_n^j$ over all $i,j,n$. Its Cartan datum is, in usual notation, 

$$
\mc F_{x}=
\begin{cases}
(\h,\h^\vee,I,J,\alpha_i,h_i,\gamma_j,c_j,b_j+\tfrac12(t/s)^2c_j),&s\ne 0\\
(\h,\h^\vee,I\sqcup \bigsqcup_{n\in\Z_{> 0}}nJ,\alpha_i,n\gamma_j,h_i,nc_j:n\in\Z_{>0}),&s=0.
\end{cases}
$$

We have decompositions
$$
\tilde{\mc G}=\tilde{\mc N}_-\oplus\tilde{\mc H}\oplus\tilde{\mc N}_+=\bigoplus_{\alpha\in\h^\vee}\tilde{\mc G}_\alpha,
$$
where $\tilde{\mc G}_\alpha$ are vector bundles of finite rank if $\alpha\ne 0$ and 
$$
\tilde{\mc H}=\h\times \C \mathbb P^1\subset \tilde\g_{\Vir,\Heis}\times\C \mathbb P^1=\tilde{\mc G}_{\Vir,\Heis}.
$$

Now, we take global Serre relations $(\operatorname{ad} e_i)^k(sL_n^j+tx_n^j)$ and the BKM Serre relations and denote the ideal family generated by them by $\tilde{\mc I}$. Let also $\tilde{\mc R}$ be the maximal ideal subfamily of $\tilde{\mc G}$ not intersecting $\tilde{\mc H}$. We say that a certain statement about bundles over $\PP^1$ is {\bfseries true for countably generic $x\in \PP^1$} if it is true for any $x\in \PP^1$ outside a certain at most countable subset of $\PP^1$.
\begin{theorem}\,\label{t:characters_of_g_in_families}
    Assume that $\mc F$ is symmetrizable and choose $\alpha\in\h^\vee$. We have
    \begin{align*}
         \tilde{\mc I}_{x,\alpha}&=\tilde{\mc I}_{0,\alpha},&\text{for any } x\in\PP^1,\\
         \tilde{\mc I}_{x,\alpha}&=\tilde{\mc I}_{0,\alpha}=\tilde{\mc R}_{0,\alpha}=\tilde{\mc R}_{x,\alpha},\quad &\text{for generic } x\in\PP^1.
    \end{align*}
        In particular, the VBKM algebra $\g_{\Vir,x}$ associated to $\mc F_x$ satisfies
        $$
        \dim\g_{\Vir,x,\alpha}=\dim\g_{\Heis,\alpha},\qquad \alpha\in\h^\vee,x\in\mathbb P^1,
        $$
         and has no nontrivial ideals not intersecting $\h$ for countably generic $x$. 
         
         %Therefore, it satisfies the usual denominator identity for BKM algebras.
\end{theorem}
We split the proof in two parts, with the second part done a bit later after additional preliminaries.
\begin{proof}[Beginning of the proof]
    First of all, $\tilde \n_{\Vir,x}\simeq \tilde\n_{\Vir,y}$ for all $x,y\in\mathbb P^1\setminus\infty$. Therefore, we have inequalities
    \begin{equation}\label{eq:inequalities_between_characters}
        \dim\g_{\Heis,\alpha}\ge \dim\g_{\Vir,x,\alpha}=\dim\g_{\Vir,\mathrm{generic},\alpha}\ge \dim\g_{\Vir,\mathrm{generic},\alpha}/\r_\alpha.
    \end{equation}
    Then, we use the fact that $\tilde I=\r$ for $\tilde\g_{\Heis,\alpha}$, so we need to compare dimensions of $\tilde{\mc R}_{0,\alpha}$ and $\tilde{\mc R}_{\mathrm{generic},\alpha}$. We will show that $\tilde{\mc R}_{\alpha}$ equals the kernel of a certain global bilinear form (Proposition \ref{p:kernel_of_Shapovalov_vs_largest_ideal}), hence its dimension is upper-semicontinuous. This implies that
    $$
    \tilde{\mc R}_{0,\alpha}\ge \tilde{\mc R}_{\mathrm{generic},\alpha},
    $$
    which proves the claim after we add \ref{eq:inequalities_between_characters} into consideration.
\end{proof}

To prove the missing part about the bilinear form, we need to recall some facts about the Shapovalov form from \cite{Sha72}. It is a bilinear form on Verma modules $\widetilde M(\lambda)$ for $\tilde\g$ and $\lambda\in\h^*$ or $\h^\vee$. First of all, we have a decomposition
$$
U(\tilde\g)=U(\h)\oplus(\tilde\n_-U(\g)+U(\g)\tilde\n_+).
$$
This decomposition defines a projection
$$
\pi:U(\tilde\g)\to U(\h).
$$
We define a bilinear form $A(\cdot,\cdot)$ on $U(\tilde\g)$ with values in $U(\h)$ by
$$
A(x,y):=\pi(\theta(x)y),
$$
where $\theta$ is the anti-automorphism of $U(\tilde\g)$ which equals $-\omega$ on the Lie algebra. For $\lambda\in\h^*$, we set $A_\lambda$ to be the evaluation of $A$ at $\lambda$. We recall the main properties of this bilinear form from \cite{Sha72}.
\begin{theorem}\label{t:Shapovalov_form_properties}
    The following is true, where orthogonality is considered with respect to $A$.
    \begin{enumerate}[label=(\roman*)]
        \item $A$ is symmetric.
        \item $U(\tilde\g)_\alpha\perp U(\tilde\g)_\alpha$ for $\alpha\ne\beta$.
        \item $U(\tilde\g)\n_+\perp U(\tilde\g)$.
        \item $A(ab,c)=A(b,\theta(a)c)$ for any $a,b,c\in U(\tilde\g)$.
        \item $A(ah,c)=hA(a,c)$ for $h\in \h$.
    \end{enumerate}
\end{theorem}

Because of properties $(iii)-(iv)$, the form $A$ descends to a $\theta$-equivariant form on the {\bfseries universal Verma module} $\widetilde M(\h^*)$ defined as
$$
M(\h^*):=\mathrm{Ind}_{U(\tilde\b_+)}^{U(\tilde\g)}U(\h)\simeq \mathrm{Ind}_{U(\tilde\n_+)}^{U(\tilde\g)}\C\simeq U(\tilde\g)/U(\tilde\g)\n_+.
$$
Because of property $(v)$, $A_\lambda$ descends to a $\theta$-equivariant form on usual Verma modules
$$
\widetilde M(\lambda)=\widetilde M(\h^*)\otimes_{U(\h)}\C_\lambda\simeq \mathrm{Ind}_{U(\tilde\b_+)}^{U(\tilde\g)}\C.
$$

% To define it, fix an isomorphism $\widetilde M(\lambda)\simeq U(\tilde\n_-)$ and write $\pi_\lambda$ for the projection 
% $$
% \pi_\lambda: U(\tilde\g)\simeq U(\tilde\n_-)\otimes U(\h)\otimes U(\tilde\n_+)\to U(\h)\to\C,
% $$
% where the last map is induced by $\lambda:\h\to\C$. Using the isomorphism $\widetilde M(\lambda)\simeq U(\tilde\n_-)$, we define a bilinear form on $\widetilde M(\lambda)$ by
% $$
% (a,b)_{\lambda}:=\pi_\lambda(\theta(a)b),\qquad a,b\in U(\n_-),
% $$ 
% where $\theta$ is the anti-automorphism of $\tilde\g$ which equals $-\omega$ on the Lie algebra.
% \begin{proposition}
%     The just defined bilinear form on $\widetilde M(\lambda)$ is $\tilde\g(\mc F)$-contravariant with respect to $\theta$:
%     $$
%     (xa,b)=(a,\theta(x)b).
%     $$
% \end{proposition}
% \begin{proof}
%     First of all, note that 
%     $$
%     (\widetilde M(\lambda)_\alpha,\widetilde M(\lambda)_\beta)=0
%     $$
%     unless $\alpha=\beta$. Therefore, the contravariance is enough to check for $x\in \tilde\g_\tau$, $a\in \widetilde M(\lambda)_\alpha$, and $b\in \widetilde M(\lambda)_\beta$ such that $\tau+\alpha=\beta$.

%     Let us first assume that $\tau<0$. Then $\beta-\tau=\alpha< 0$.
    
%     We know that elements of $\tilde\n_-$ act by left multiplications and elements of $\n_+$ act by algebra derivations. Then, for $x\in \tilde\n_-$ we have
%     $$
%     (xa,b)=\pi_\lambda(\theta(xa)b)=\pi_\lambda(\theta(a)\theta(x)b)=(a,\theta(x)b).
%     $$

%     Now, assume that $\tau>0$. Then 
% \end{proof}

\begin{proposition}\label{p:kernel_of_Shapovalov_vs_largest_ideal}
    Let $S_\lambda$ be the kernel of $A_\lambda$ on $\widetilde M(\lambda)\simeq U(\tilde\n_-)$. Then $S_\lambda$ is the largest submodule of $M(\lambda)$ not intersecting the highest weight component and 
    $$
    \bigcap_{\lambda\in\h^*}S_\lambda\cap \tilde\n_-=\r_-.
    $$
    Moreover, for any negative root $\alpha$, we can choose finitely many $\lambda_1,\ldots,\lambda_n\in\h^*$ such that
    $$
    \bigcap_{i=1}^nV_{\lambda_i}\cap \tilde\n_\alpha=\r_\alpha.
    $$
\end{proposition}
\begin{proof}
    Clearly, $S_\lambda$ is a submodule of $\widetilde M(\lambda)$ not intersecting the highest weight component. To prove that it is the largest such module, let $v\in \widetilde M(\lambda)_\mu$ be such that $v\mod S_\lambda$ is annihilated by $\tilde\n_+$. Then for any
    $$
    0=A_\lambda(U(\n_+)\n_+v,v_\lambda)=A_\lambda(v,U(\n_-)\n_-v_\lambda)=A_\lambda(v,\widetilde M(\lambda)_{<\lambda}),
    $$
    where $v_\lambda$ is the highest weight vector of $\widetilde M(\lambda)$. Then $v\in S_\lambda$ or $v\in \widetilde M(\lambda)_\lambda$, which implies that $\widetilde M(\lambda)/S_\lambda$ does not contain any highest vector except $v_\lambda$. This shows that it is irreducible, so we are done with the first part.

    Take $x\in\r_-$. Then $\theta(x)\in\r_+$. Since $\r$ does not intersect $\h$, $[\theta(x),U(\tilde\g)]$ has no $U(\h)$-component in the PBW decomposition. This implies that for any $a\in U(\tilde\n_-)$,
    $$
    A_\lambda(x,a)=\pi_\lambda(\theta(x)a)=\pi_\lambda([\theta(x),a])=0.
    $$
    This shows that $x\in V\cap \tilde\n_-$, showing the first part of the statement.
    
    Now, take $x\in \bigcap_{\lambda\in\h^*}S_\lambda\cap \tilde\g_\alpha$ for some negative root $\alpha$. We have 
    $$
    \pi_\lambda(\theta(a)x)=0\qquad \text{for all }a\in U(\tilde\n_-).
    $$
    Equivalently,
    $$
    \pi_\lambda(ax)=0\qquad \text{for all }a\in U(\tilde\n_+)_\alpha.
    $$
    Therefore, the $U(\h)$-components of elements in $U(\tilde\n_+)_\alpha x$ lie in $U(\ker\lambda)$. Since
    $$
    \bigcap_{\lambda\in\h^*}U(\ker\lambda)=0,
    $$
    we get that the $U(\h)$-components of elements in $U(\tilde\n_+)x$ are zero.

    Suppose that there exists $y\in\tilde\n_+$ such that $0\ne[y,x]\in\h\oplus\tilde\n_+$. Since $\tilde \n_+$ is generated by simple root spaces and the only way to reach positive roots from negative roots by applying simple roots and not leaving the roots cone is by passing through $\h$, there exists $z\in\tilde\g_{-\alpha}$ such that $0\ne[x,z]\in\h$. But then
    $$
    \lambda([x,z])=\pi_\lambda([x,z])=-\pi_\lambda(zx)=-A_\lambda(\theta(z),x)=0
    $$
    for any $\lambda$. This is a contradiction, which shows that there is no $y$ as we assumed. Therefore, the ideal of $\tilde\g$ generated by $x$ is contained in $\tilde\n_-$, and therefore is contained in $\r_-$. This finishes the proof.

    To get the last statement, notice that by finite-dimensionality of $\tilde \g_\alpha$, the space $[\g_\alpha,\g_{-\alpha}]\subset \h$ is finite-dimensional, and the above proof works if we just intersect over some finite set $\alpha_1,\ldots,\alpha_n$ such that 
    $$
    [\tilde\g_\alpha,\tilde\g_{-\alpha}]\cap\bigcap_{i=1}^n\ker\alpha_i=0.
    $$
\end{proof}

\begin{proof}[End of the proof of Theorem \ref{t:characters_of_g_in_families}]
    Choose $\lambda_1,\ldots,\lambda_n\in\h^*$ such that
    $$
    [\tilde{\mc G}_{x,\alpha},\tilde{\mc G}_{x,\alpha}]\cap\bigcap_{i=1}^n\ker\alpha_i=0
    $$
    for any $x\in \PP^1$. We can do it since after we have chosen $\lambda_i$ such that this condition is satisfied for one $x$, it is satisfied for an open neighborhood of $x$ since the condition is generic. Then there are at most finitely many points remaining when the condition is not satisfied, and we just add more $\lambda_i$ to our collection to satisfy the condition at those remaining points.
    
    By Proposition \ref{p:kernel_of_Shapovalov_vs_largest_ideal}, 
    $$
    \mc R_{x,\alpha}=\bigcap_{i=1}^n \ker(\cdot,\cdot)_{\lambda_i}\cap \tilde\n_{x,\alpha}.
    $$
    Since the intersection is finite, the dimension of $\mc R_{x,\alpha}$ is upper semi-continuous. This finishes the missing part of the proof of Theorem \ref{t:characters_of_g_in_families}, so we are done.
\end{proof}

Recall the last condition in Definition \ref{def:enh_gen_Cartan_datum}: 
\begin{equation}\label{eq:special_Virasoro_condition}
    \text{$\alpha(c_j)=0$ implies $\alpha(b_j)=0$ for any $\alpha\in\{\alpha_i,\gamma_k:i\in I,k\in J\setminus j\}$. }
\end{equation}

Since $b_j$ deforms to $b_j+\frac{t^2}{2}c_j$ in the $\PP^1$-family of VBKM algebras that we are studying, a countably generic VBKM algebra in this family satisfies a stronger condition: 
\begin{equation}\label{eq:generic_Virasoro_condition}
    \text{$\alpha(c_j)=0$ if and only if $\alpha(b_j)=0$ for any $\alpha\in\{\alpha_i,\gamma_k:i\in I,k\in J\setminus j\}$. }
\end{equation}

Because of Computation \ref{comp:why_not_submodules_of_Vermas} and Proposition \ref{p:new_ideal_in_special_VBKM}, we can state the following conjecture.
\begin{conjecture}\label{conj:char_of_g}
    The statement of Theorem \ref{t:characters_of_g_in_families} holds for any $x\in\PP^1$ for which the corresponding VBKM algebra satisfies (\ref{eq:generic_Virasoro_condition}). In other words, the VBKM algebra $\tilde\g(\mc F_x)$ contains no ideals not intersecting $\h$.
\end{conjecture}

Now, we turn to the representation theory of $\g_{\Vir}$. But first, a little warm-up example.
\begin{example}
    For $c,h\in\C$, define the family over Verma modules over $\mc V$ by
    $$
    \mc M(c,h):=\mathrm{Ind}_{\mc V_{\ge 0}}^{\mc V}\mc O_{c,h},
    $$
    where $\mc V_{> 0}$ acts trivially on $\mc O_{c,h}$, while $\mc C$ acts by $c$ and $\mc L$ acts by $h$.

    For $x=[s:t]\in\PP^1$, recall that the Cartan part of $\mc V_{\ge 0}$ equals $c$ and $L_0+\frac{t^2}{2}c$. In particular, the module $\mc M(c,h)_x$ over $\mc V_x$ is isomorphic to $M(c,h+\frac{t^2}{2s^2}c)$ if $s\ne 0$ and to $M(c)$ if $s=0$ with $L_0$ acting by $h$. If $c\ne 0$, $M(c)$ is irreducible, in which case by semi-continuity of weight space dimensions we expect $M(c,h+\frac{t^2}{2s^2}c)$ to be irreducible for countably generic $x$. We can easily see from (\ref{eq:Virasoro_highest_weights}) that it is indeed true.
\end{example}

Motivated by this example, we consider the family of Verma modules $\widetilde{\mc M}(\lambda)_x$ over $\mc G_x$. For any such module, let $\mc J_{x,\lambda}$ be the submodule generated by Serre relations
\begin{align*}
    f_i^{1+2\lambda(h_i)/a_{ii}}.1&,\qquad i\in I_{>0},\, 2\lambda(h_i)/a_{ii}\in\Z_{\ge 0},\\
    f_i.1&,\qquad i\in I, \,\lambda(h_i)=0,\\
    L_n^j.1&,\qquad \lambda(c_j)=\lambda(b_j)=0, \,n\in\Z
\end{align*}
and $\tilde{\mc I}_{x,\lambda}\widetilde{\mc M}(\lambda)_x$. It follows from Proposition \ref{p:kernel_of_Shapovalov_vs_largest_ideal} that $\mc S_{x,\lambda}$ is the largest submodule of $\widetilde{\mc M}(\lambda)_x$ not intersecting the highest weight space (the local versions are denoted by $S_\lambda$ and $J_\lambda$). Denote
$$
\mc L(\lambda):=\widetilde{\mc M}(\lambda)/\mc J_\lambda,\qquad \mc L'(\lambda):=\widetilde{\mc M}(\lambda)/\mc S_\lambda,
$$
and their corresponding local versions by $L(\lambda)$ and $L(\lambda)'$.

\begin{definition}
    We say that $\lambda\in \h^*$ is {\bfseries admissible dominant integral} if it satisfies
\begin{itemize}
    \item {\bfseries Admissibility}: for any $j\in J$, $\lambda(c_j)=0$ if and only if $\lambda(b_j)=0$.
    \item {\bfseries Dominance}: for any $i\in I$ and $j\in J$, $\lambda(h_i)\ge 0$ and $\lambda(c_j)\ge 0$.
    \item {\bfseries Integrability}: for any $i\in I_{>0}$, $2\lambda(h_i)/a_{ii}\in \Z_{\ge 0}$.
\end{itemize}
\end{definition}

Recall the following result from \cite[Theorem 3.17]{Jur96}.
\begin{theorem}
    For any symmetrizable BKM algebra, $S_\lambda=J_\lambda$, and therefore $L(\lambda)=L'(\lambda)$. Moreover, there is a Weyl-Kac type character formula for the character of $L(\lambda)$.
\end{theorem}

Using the interpretation of $S_\lambda$ as the kernel of the Shapovalov form in Proposition \ref{p:kernel_of_Shapovalov_vs_largest_ideal}, we get the following result with the same proof as of Theorem \ref{t:characters_of_g_in_families}.

\begin{theorem}\,\label{t:characters_of_reps_in_families}
    Assume that $\mc F$ is symmetrizable, and choose admissible dominant integral $\lambda\in\h^*$. We have
    \begin{align*}
         \tilde{\mc J}_{x,\lambda}&=\tilde{\mc J}_{0,\lambda},&\text{for any } x\in\PP^1,\\
         \tilde{\mc J}_{x,\lambda}&=\tilde{\mc J}_{0,\lambda}=\tilde{\mc S}_{0,\lambda}=\tilde{\mc S}_{x,\lambda},\quad &\text{for generic } x\in\PP^1.
    \end{align*}
        In particular, 
        $$
        \dim\mc L(\lambda)_{x,\mu}=\dim\mc L(\lambda)_{0,\mu},\qquad \mu\in\h^*,x\in\mathbb P^1,
        $$
        and these modules are irreducible for countably generic $x$. 
         
         %Therefore, it satisfies the usual denominator identity for BKM algebras.
\end{theorem}
\begin{remark}
    Nore that the statement analogous to Conjecture \ref{conj:char_of_g} is certainly not true in this case since Verma modules for a single Virasoro algebra (which is a particular example of a VBKM algebra) are not always irreducible.
\end{remark}

\begin{proposition}
    Assume that $\lambda$ is admissible dominant integral. The operator $n_i$ from \ref{p:Weyl_group_on_g} acts on $\mc L(\lambda)_x$ and sends the $\mu$-weight space to the $s_i(\mu)$-weight space. In particular, dimensions of weight spaces are invariant under the Weyl group action.
\end{proposition}
\begin{proof}
    Similar to the proof of \ref{p:Weyl_group_on_g}, see \cite[Lemma 3.8]{Kac90}.
\end{proof}

% \begin{theorem}
%     For a countably generic choice of $b_i$ in the Cartan datum $\mc F$, the irreducible representation $V(\lambda)$ of $\g_{\Vir}$ of highest weight $\lambda$ has the same character as the similar representation for the corresponding $\g_{\Heis}$. In particular, if $\lambda$ is integral with respect to the Cartan datum of $\g_{\Heis}$, then the character of $V(\lambda)$ satisfies the Weyl-Kac character formula.
% \end{theorem}
% \begin{proof}
    
% \end{proof}

\subsection{Filtration on the positive part of VBKM algebras}
We can give another connection between $\g_{\Vir}$ and $\g_{\Heis}$. Introduce a filtration on $U(\tilde\n_{\Vir,-})$ by setting $\deg f_i=\deg L_n^j=1$ for all $i\in I,j\in J$ and $n<0$ and the property that multiplication preserves this filtration. This also gives a filtration on all Verma modules $M(\lambda)$ compatible with the filtration on $U(\tilde\n_-)$.
\begin{theorem}\label{t:ass_graded_of_Vir_is_Heis}
    There are isomorphisms
    \begin{align*}
        \tilde f:U(\tilde\n_{\Heis,-})&\xrightarrow{\sim} \operatorname{gr}U(\tilde\n_{\Vir,-}),\\
        \widetilde M_\Heis(\lambda)&\xrightarrow{\sim} \operatorname{gr}\widetilde M_\Vir(\lambda)
    \end{align*}
    sending $f_i$ to $f_i$ and $h_n^j$ to $L_n^j$. The isomorphisms of modules are as of modules over $\tilde\n_-$.
    
    They the ideal (submodule) generated by Serre relations into the ideal (submodule) generated by Serre relations and descends to isomorphisms
    \begin{align*}
         f:U(\n_{\Heis,-})&\xrightarrow{\sim} \operatorname{gr}U(\n_{\Vir,-}),\\
        M_\Heis(\lambda)&\xrightarrow{\sim} \operatorname{gr}M_\Vir(\lambda),\\
        L_\Heis(\lambda)&\xrightarrow{\sim} \operatorname{gr}L_\Vir(\lambda).
    \end{align*}
    The isomorphisms of modules are as of modules over $\n_-$
\end{theorem}
\begin{proof}
    We do the proof for $U(\tilde \n_-)$, as the proofs for the modules follow the same arguments. Since $\deg L_n^j=\deg L_m^j=1$ and $\deg[L_n^j,L_m^j]=\deg L_{m+n}=1$, the associated elements $\bar L_n^j$ and $\bar L_m^j$ in the first graded component of $\operatorname{gr}U(\tilde\n_{\Vir,-})$ commute, hence generate the positive part of a Heisenberg algebra. Since the associated graded construction commutes with free products of algebras, we get the desired isomorphism $\tilde f$.

    It is easy to see from the form of Serre relations that they are preserved by $\tilde f$. This proves that $\tilde f$ descends to a surjective map $f$. By Theorem \ref{t:characters_of_g_in_families} (Theorem \ref{t:characters_of_reps_in_families} for modules), the characters of the domain and codomain of $f$ are the same, which implies that $f$ is an isomorphism.
\end{proof}

% Then Conjecture \ref{conj:char_of_g} is equivalent to
% \begin{conjecture}
%     The map $f$ from Lemma \ref{l:ass_graded_of_Vir_is_Heis} is an isomorphism.
% \end{conjecture}

% \begin{lemma}
%     Let $\g^1,\g^2,\g^3$ be Lie algebras graded by some lattice $L$. Assume that $\ch\g^2=\ch\g^3$, and let $e_i^1,e_i^2$ be graded bases of $\g^2$ and $\g^3$ such that $e_i^1$ and $e_i^2$ have the same degree for any $i$. Let $I^1$ be an ideal of $\g^1*\g^2$ generated by 
%     $$
%     \ch(\g_1*\g_2)=\ch(\g_1*\g_3),
%     $$
%     where $*$ denotes the free product.
% \end{lemma}
% \begin{proof}

% \end{proof}

\section{Examples}\label{s:examples}
\subsection{Lattice vertex operator algebras}
Let $L$ be an even lattice, i.e. a free abelian group with a non-degenerate bilinear product $(\cdot,\cdot)$ with
$$
\alpha^2:=(\alpha,\alpha)\in2\Z,\qquad \alpha\in L.
$$

Denote by $L_\C$ the complexification of $L$. To $L$, there is an associated {\bfseries lattice vertex operator algebra (lattice VOA)} $V_L$. As a vector space,
$$
V_L=S(t^{-1}L_\C[t^{-1}])\otimes \C[L]=\bigoplus_{\alpha\in L}S(t^{-1}L_\C[t^{-1}])\ket\alpha,
$$
where by $\ket\alpha$ we denote the element of the group algebra $\C[L]$ corresponding to $\alpha\in L$. We denote the summand $S(t^{-1}L_\C[t^{-1}])\ket\alpha$ by $V_{L,\alpha}$.

We briefly describe the necessary structure on $V_L$ and refer the reader to \cite[Section 5.4]{FBZ04} and \cite[Section 5]{Kac98} for details. For a power series $f\in W\jet{z,z^{-1}}$ for some vector space $W$, we write
$$
\int f(z)\,dz
$$
for the $z^{-1}$-coefficient of $f$. In particular,
$$
\int z^nf(z)\,dz
$$
returns the $z^{-1-n}$-coefficient of $f$. Because of this, we write
$$
f=\sum_{n\in\Z}f_{n}z^{-n-1},\qquad f_n=\int z^nf(z)\,dz.
$$

One needs to choose a cocycle $c:L\times L\to\{\pm 1\}$. Its exact form and meaning will not be important for our computations, see \cite[Section 5.4]{FBZ04} for details on this. We extend this to a pairing between $L$ and $V_L$ given by
$$
c(\alpha,Q\ket\beta):=c(\alpha,\beta)Q\ket\beta.
$$
We also introduce operators $e^\alpha$ on $V_L$ for $\alpha\in L$ as
\begin{align*}
    e^\alpha(Q\ket\beta):=Q\ket{\alpha+\beta}.
\end{align*}

In $S(t^{-1}L_\C[t^{-1}])$, we denote elements $\alpha\otimes t^{-n}$ by $\alpha(-n)$. These also act as operators on $V_L$ by left multiplication. We can extend these operators to a representation of the Heisenberg algebra associated to $L$ letting $\alpha(n)$ act by derivations, central element act by $1$, and $\alpha(0)$ act on $Q\ket\beta$ by $(\alpha,\beta)$. This turns each $V_{L,\alpha}$ into the Verma module for this Heisenberg algebra with the $0$-element $\alpha(0)$ acting by $(\alpha,\beta)$.

We define the vertex operator $x_n$ for $x\in V_L$ in the following way:
\begin{align*}
    (\alpha(-m)\ket 0)_n&:=\binom{n}{m-1}\alpha(n-m-1),\\
    \ket\alpha_n&:=\int e^\alpha c(\alpha,-) z^{n+(\alpha,-)}\exp\left(\sum_{n>0}\frac{\alpha(-n)}nz^n\right)\exp\left(-\sum_{n>0}\frac{\alpha(n)}nz^{-n}\right),\\
    (a_{-1}b)_n&=\sum_{k\ge 0}a_{n-k-1}b_k+\sum_{k<0}b_{k}a_{n-k-1}.
\end{align*}
These rules define $x_n$ recursively for any $x\in V_L$. There are lots of other relations between these operators called operator product expansions or Borcherds identities, to which we refer to \cite{FBZ04} again. The construction we just gave will be enough for our purposes.

\begin{remark}
    In future computations, we will not care about the cocycle since we will do our computations on elements in $V_{L,\alpha}$, hence we can manipulate the signs by changing these elements by their negatives. For example, replacing the central charge in $\Heis$ by its negative, we get relations
    $$
    [x_n,x_m]=-n\delta_{m,-n}c.
    $$

    Alternatively, we have flexibility with choosing the cocycle and for any two elements $\alpha,\beta\in L$, we can always choose the cocycle such that $c(\alpha,\beta)=1$ (for example, by setting $c'(\alpha,\beta):=c(\beta,\alpha)$). See \cite[Remark 5.5a]{Kac98} for details.
\end{remark}

Another important ingredient is the Virasoro algebra of operators acting on $V_L$. Choose dual bases $h_i$ and $h^i$ of $L_\C$ and define
$$
w:=\frac12\sum_ih_i(-1)h^i(-1).
$$
We recall some crucial facts about this vector.
\begin{theorem}\,\label{t:conformal_structure_properties}
    \begin{enumerate}
        \item The operators $L_n:=w_{n+1}$ satisfy Virasoro relations with central charge $\operatorname{rk} L$. If $P$ is a homogeneous element of $S(t^{-1}L_\C[t^{-1}])$ with respect to the $t$-grading, then
        $$
        L_0(P\ket\alpha)=\deg P+\frac{\alpha^2}{2}\in\Z.
        $$
        \item If $L_0x=ax$ for $a\in\Z$, then $[L_0,x_n]=(1+a-n)x_n$.
        \item $V_L/L_{-1}V_L$ and its Virasoro degree one piece $V_{L,1}/L_{-1}V_{L,0}$ is a Lie algebra with the Lie bracket given by $[x,y]:=x_0y$.
    \end{enumerate}
\end{theorem}

\subsection{Lattice VOA realizations of VBKM algebras}
First, recall the classical theory. Let $A$ be a symmetric Cartan matrix with the associated Kac-Moody algebra $\g$. Let $L$ be the root lattice of $\g$ with basis of simple roots $\alpha_i$ and bilinear form given by the matrix $A$ in this basis.

\begin{theorem}
    The map
    \begin{align*}
        e_i&\mapsto \ket\alpha_i,\\
        h_i&\mapsto \alpha_i(-1)\ket 0,\\
        f_i&\mapsto \ket\alpha_i
    \end{align*}
    gives a homomorphism of Lie algebras $\g\to V_{L,1}/L_{-1}V_{L,0}$.
\end{theorem}
\begin{proof}
    The relations in $\tilde\g(A)$ are straightforward to check. For Serre relations, let $(\alpha_i,\alpha_j)=-N$, and consider 
    $$
    x:=(\operatorname{ad}e_i)^{1+N}e_j=P\ket{(1+n)\alpha_i+\alpha_j}.
    $$
    We have
    \begin{align*}
        x=L_0x&=\left[\deg P+\frac{((1+N)\alpha_i+\alpha_j)^2}{2}\right]x\\
        &=\left[\deg P+\frac{2(1+N)^2-2N(1+N)+2}{2}\right]x=(\deg P+N+2)x.
    \end{align*}
    Since $\deg P+N+2>1$, we must have $x=0$, which checks Serre relations.
\end{proof}

We introduce the following construction. For $0\ne \tau\in L$ with $\tau^2=0$ and $a\in L_\C$, define operators
$$
 L_n^{\tau,a}:=(a(-1)\ket{n\tau})_0,\qquad n\in\Z.
$$
\begin{proposition}
    Assume that the cocycle is chosen so that $c(\tau,\tau)=1$. Then the operators $L_n^\tau$ satisfy relations
    \begin{equation}\label{eq:relations_between_pseudo_Virasoro_elements}
[ L_n^{\tau,a}, L_m^{\tau,a}]=-(a,\tau)(n-m) L_{m+n}^{\tau,a}+\delta_{n,-m}((a,\tau)^2n^3+(a,a)n)\tau(0).
    \end{equation}
\end{proposition}
\begin{proof}
    We have $(n\tau,m\tau)=0$, so
    \begin{align*}
        \ket{n\tau}_ka(-1)\ket{m\tau}&=\int z^{k}\exp\left(\sum_{s>0}\frac{n\tau(-s)}sz^{s}\right)\left(1-n\tau(1)z^{-1}\right)a(-1)\ket{m\tau}\\
        &=\begin{cases}
            0,&k\ge 1,\\
            n(\tau,a)\ket{(n+m)\tau},&k=0
        \end{cases}
    \end{align*}
    and
    $$
    \ket{n\tau}_k\ket{m\tau}=
    \begin{cases}
        0,&k\ge 0,\\
        \ket{(n+m)\tau},&k=-1,\\
        n\tau(-1)\ket{(n+m)\tau},&k=-2,
    \end{cases}
    $$
    This gives
    \begin{align*}
        (a(-1)\ket{n\tau})_0&a(-1)\ket{m\tau}=\sum_{k\ge 0}a(-1-k)\ket{n\tau}_ka(-1)\ket{m\tau}+\sum_{k\ge 0}\ket{n\tau}_{-1-k}a(k)a(-1)\ket{m\tau}\\
        &=(-n(a,\tau)a(-1)+m(a,\tau)a(-1)+(-(a,\tau)^2mn^2+(a,a)n)\tau(-1))\ket{(m+n)\tau}.
    \end{align*}
    Since 
    $$
    (m+n)\tau(-1)\ket{(m+n)\tau}=T(\ket{(m+n)\tau}),
    $$
    we have
    $$
    [ L_n^{\tau,a}, L_m^{\tau,a}]=-(a,\tau)(n-m) L_{m+n}^{\tau,a}+\delta_{n,-m}((a,\tau)^2n^3+(a,a)n)\tau(0),
    $$
    as desired.
\end{proof}

So, the elements $ L_n^{\tau,a}$ satisfy Virasoro/Heisenberg relations under certain rescalings and adding some multiple of $\tau(0)$ to $L_0$. The precise statement is the following.
\begin{corollary}\
    \begin{enumerate}[label=(\roman*)]
        \item If $(a,\tau)=0$, then $\tilde L_n^{\tau,a}$ and $\tau(0)$ span the Heisenberg algebra with central element $(a,a)\tau(0)$.
        \item If $(a,\tau)\ne 0$, then the elements 
        $$
        \tilde L_n^{\tau,a}:=
        \begin{cases}
            -L_n^{\tau,a}/(a,\tau),&n\ne 0,\\
            -L_0^{\tau,a}/(a,\tau)+(1+(a,a)/(a,\tau))\tau(0),&n=0
        \end{cases}
        $$
        and $\tau(0)$ span the Virasoro algebra with central charge $12\tau(0)$.
        \item If $(a,a)=(a,\tau)=-1$, then $L_n^{\tau,a}$ form the standard Virasoro basis with central charge $12\tau(0)$.
    \end{enumerate}
\end{corollary}
\begin{proof}
    Straightforward from the commutativity relation (\ref{eq:relations_between_pseudo_Virasoro_elements}).
\end{proof}

Let us take two such Virasoro algebras: $\Vir^{\tau_1,a_1}$ and $\Vir^{\tau_2,a_2}$ with $(a_i,a_i)=(a_i,\tau_i)=-1)$.
\begin{proposition}\label{p:relations_between_Virasoros_in_lattice_VOAs}
    Assume that $(\tau_1,\tau_2)<0$ and the chosen cocycle satisfies $c(\tau_1,\tau_2)=1$. Then
    \begin{enumerate}[label=(\roman*)]
        \item $[{\Vir_{\ge 2}^{\tau_1,a_1},\Vir_{\le -1}^{\tau_2,a_2}}]=[{\Vir_{\ge 1}^{\tau_1,a_1},\Vir_{\le -2}^{\tau_2,a_2}}]=0$, and the symmetric condition holds
        \item If $(\tau_1,\tau_2)<-1$, then $[{\Vir_{\ge 1}^{\tau_1,a_1},\Vir_{\le -1}^{\tau_2,a_2}}]=0$ and the symmetric condition holds. In this case, $\Vir^{\tau_1,a_1}$ and $\Vir^{\tau_2,a_2}$ generate a VBKM algebra.
        \item If $(\tau_1,\tau_2)=-1$, then $[L_{1}^{\tau_1,a_1},L_{-1}^{\tau_2,a_2}]=\pm \ket{\tau_1-\tau_2}$ and $[L_{-1}^{\tau_1,a_1},L_{1}^{\tau_2,a_2}]=\pm \ket{\tau_2-\tau_1}$.
        \item If $(\tau_1,\tau_2)=-1$, then the $\sl_2$ triple $e,h,f=\ket{\tau_1-\tau_2},\tau_1(0)-\tau_2(0), \ket{\tau_2-\tau_1}$ and $\Vir^{\tau_2,a_2}$ generate a VBKM algebra containing $\Vir^{\tau_1,a_1}$. With
        $$
        (h,c,b)=(\tau_1-\tau_2,\tau_2,a_1),
        $$
        it has the Cartan matrix
        $$
        \begin{pmatrix}
            2&-1\\
            -1&0\\
            1+(a_2,\tau_1)&-1
        \end{pmatrix}.
        $$
    \end{enumerate}
\end{proposition}
\begin{proof}
    Let $n\ge 0$ and $m\le 0$. We have $(n\tau_1,m\tau_2)\ge -mn$, so
    \begin{align*}
        \ket{n\tau_1}_ka_2(-1)\ket{m\tau_2}&=\int z^{k+mn(\tau_1,\tau_2)}\exp\left(\sum_{s>0}\frac{n\tau_1(-s)}sz^{s}\right)\left(1-n\tau_1(1)z^{-1}\right)a_2(-1)\ket{m\tau_2}\\
        &=\begin{cases}
            0,&k+mn(\tau_1,\tau_2)\ge 1,\\
             n(\tau_1,a_2)\ket{n\tau_1+m\tau_2},&k+mn(\tau_1,\tau_2)=0.
        \end{cases}
    \end{align*}
    Similarly,
    $$
    \ket{n\tau_1}_k\ket{m\tau_2}=\begin{cases}
        0,&k+mn(\tau_1,\tau_2)\ge 0,\\
        \ket{m\tau_1+n\tau_2},&k+mn(\tau_1,\tau_2)=-1.
    \end{cases}
    $$
    
    Using this and that $k+mn(\tau_1,\tau_2)\ge k+1$ with our assumptions, we get
    \begin{align*}
        (a_1(-1)\ket{\tau_1})_0a_2(-1)\ket{\tau_2}&=\sum_{k\ge 0}a_1(-1-k)\ket{n\tau_1}_ka_2(-1)\ket{m\tau_2}+\sum_{k\ge 0}\ket{n\tau_1}_{-1-k}a_1(k)a_2(-1)\ket{m\tau_2}\\
        &=0\pm(a_1,a_2)\ket{n\tau_1}_{-2}\ket{m\tau_2},
    \end{align*}
    which is zero if $mn\le -2$ and $\pm \ket{m\tau_1+n\tau_2}$ if $n=-m=1$ and $(\tau_1,\tau_2)=-1$. This proves the entire Proposition.
\end{proof}

\begin{example}
    Let $L$ be a lattice with basis $c,d$ with Gram matrix
    $$
    \mat0{-1}{-1}0.
    $$
    We have two Virasoro algebras inside: $\Vir^c$ and $\Vir^d$ given by $(d(-1)+\tfrac12c(-1))\ket{nc}$ and $(c(-1)+\tfrac12d(-1))\ket{nd}$, respectively. This is the case $(iii)-(iv)$ of Proposition \ref{p:relations_between_Virasoros_in_lattice_VOAs}. The Cartan matrix is
    $$
    \begin{pmatrix}
            2&-1\\
            -1&0\\
            1/2&-1
    \end{pmatrix}.
    $$
\end{example}

\begin{example}\label{e:special_VBKM_realization}
    In this example, we construct a quotient of a VBKM algebra satisfying relations from Proposition \ref{p:new_ideal_in_special_VBKM}.

    Let $L$ be an even sublattice of the vector space with basis $c,d,a,b$ with Gram matrix
    $$
    \begin{pmatrix}
        0&-2&-1&0\\
        -2&0&0&-1\\
        -1&0&-1&0\\
        0&-1&0&-1
    \end{pmatrix}.
    $$
    For example, we can set $L:=\lge2a,2b,c,d\rge$. Take $\Vir^{a,c}$ and $\Vir^{b,d}$. We compute $[L^{a,c}_{n},L^{b,d}_{1}]$ for $n>0$:
    \begin{align*}
        (a(-1)\ket{nc})_0b(-1)\ket{d}&=\sum_{k\ge 0}a(-1-k)\ket{nc}_kb(-1)\ket{d}+\sum_{k\ge 0}\ket{nc}_{-1-k}a(k)b(-1)\ket{d}.
    \end{align*}
    Since $(a,b)=(a,d)=0$, we have $a(k)b(-1)\ket{d}=0$ for all $k\ge 0$. Since $(nc,d)\le -2$ and $(b,c)=0$, we have
    $$
    \ket{nc}_kb(-1)\ket{d}=
    \begin{cases}
        (b,c)\ket{c+d}=0,&n=1,k=0,\\
        0,&k>0.
    \end{cases}
    $$
    Thus, $[L^{a,c}_{n},L^{b,d}_{1}]=0$ for all $n\ge 0$, showing that the desired relations are satisfied.

    Notice that if we did not assume $(a,b)=0$, we would get
    \begin{align*}
    [L^{a,c}_{1},L^{b,d}_{1}]&=\ket{c}_{-2}a(1)b(-1)\ket{d}\\
    &=(a,b)\ket{c}_{-2}\ket d\\
    &=\frac16(a,b)\left[c(-1)^3+3c(-1)c(-2)+2c(-3)\right]\ket{c+d}\ne 0.
    \end{align*}
    So, there is a non-trivial ideal in the lattice VOA realization if $(a,b)\ne 0$.
\end{example}

Now, we give examples of situations when the assumption $\gamma_j(h_i)=0$ iff $\beta_j(h_i)=0$ or $\gamma_j(c_k)=0$ iff $\beta_j(c_k)=0$ fails.
\begin{antiexample}
    Let $L$ be a lattice with basis $\alpha,c,d$ with Gram matrix
    $$
    \begin{pmatrix}
            2&0&0\\
            0&0&-1\\
            0&-1&0
    \end{pmatrix}.
    $$
    Let us consider $\Vir^c$ given by $(d(-1)+\tfrac12c(-1))\ket{nc}$ and $\operatorname{Heis}^c$ given by $\alpha(-1)\ket{nc}$. We have $\gamma(c)=(c,c)=0$ but $b(c)=(d+\tfrac12c,c)=-1$. Let us compute the commutators:
    \begin{align*}
        (\alpha(-1)\ket{kc})_0(d(-1)+\tfrac12c(-1))\ket{nc}=-k\alpha(-1)\ket{(k+n)c},
    \end{align*}
    which shows that the algebra generated by $\Vir^c$ and $\operatorname{Heis}^c$ is isomorphic to $\Vir^c\ltimes\operatorname{Heis}^c$ studied before.
\end{antiexample}
\begin{antiexample}
    Let $L$ be the lattice from the previous example.
    Let us consider $\Vir^c$ given by $(d(-1)+\tfrac12c(-1))\ket{nc}$ and $\sl_2$ given by $\ket{c-\alpha}, c-\alpha,\ket{\alpha-c}$. The root of this $\sl_2$ is orthogonal to the central charge but not to $L_0$. Let us compute the commutators:
    \begin{align*}
        ((d(-1)+\tfrac12c(-1))\ket{nc})_0\ket{c-\alpha}&=\pm\ket{(n+1)c-\alpha},\\
        ((d(-1)+\tfrac12c(-1))\ket{nc})_0\ket{\alpha-c}&=\pm\ket{(n-1)c-\alpha},
    \end{align*}
    where the $\pm$ signs come from the cocycle $c(-,-)$ chosen for the lattice VOA and do not matter for us much.
    
    We can see that we get the root system of $\widehat\sl_2$, the affine $\sl_2$. Therefore, the algebra generated by $\Vir^c$ and $\sl_2$ equals $\Vir^c\ltimes \widehat\sl_2$. If we think of $\widetilde\sl_2$ as a central extension of the algebra $\operatorname{Maps}(\C[t,t^{-1}],\sl_2)$, then the Virasoro action comes from the Witt algebra action on $\C[t,t^{-1}]$ by derivations.
\end{antiexample}

\subsection{Fake Monster Lie algebra}
We recall main facts about the fake monster Lie algebra introduced by Borcherds in \cite{Bor90}. We follow the exposition in \cite[Section 1.4]{Nie02}.

Let $V$ be the lattice VOA associated to the even unimodular Lorentzian lattice $II_{25,1}$ of rank $26$. There is a unique such lattice, and it can be represented as $\Lambda\oplus II_{1,1}$, where $\Lambda$ is the Leech lattice and $II_{1,1}$ is the unique even unimodular Lorentzian lattice of rank $2$. Let $\rho,\sigma$ be a basis of $II_{1,1}$ with Gram matrix
$$
\mat{0}{-1}{-1}{0}.
$$
We write an element $\lambda+m\sigma+n\rho\in II_{25,1}$ for $\lambda\in \Lambda$ as $(\lambda,m,n)$.
\begin{definition}
    The {\bfseries space $P_k$ of  physical (or primary) states of degree $k$} is defined as
    $$
    P_k:=\{v\in V_L:\forall_{n>0}\,L_nv=0, L_0v=kv\}.
    $$
    The {\bfseries Lie algebra of physical states} is $P_1/L_{-1}P_0$, where $L_{-1}$ is the Virasoro mode from Theorem \ref{t:conformal_structure_properties}.
\end{definition}

We need one more piece of notation: for a positive integer $k$, denote by $p^{(k)}(n)$ the number of $k$-colored partitions of $n$. In other words, $p^{(1)}(n)$ is the number of partitions of $n$ and the generating functions are related by
$$
\sum_{n}p^{(k)}(n)q^n=\left(\sum_{n}p^{(1)}(n)q^n\right)^k.
$$
\begin{theorem}\,\label{t:fake_monster_properties}
    \begin{enumerate}[label=(\roman*)]
        \item The zero mode multiplication $[a,b]:=a_0b$ in $V$ descends to $P_1/L_{-1}P_0$ making it a Lie algebra.
        \item There is an invariant bilinear form $(\cdot,\cdot)$ on $P_1/L_{-1}P_0$. We denote the quotient of $P_1/L_{-1}P_0$ by its kernel by $\g$, which is a Lie algebra.
        \item $\g$ is a BKM Lie algebra with
        \begin{align*}
            I_{>0}&=\{(\lambda,1,1-\lambda^2/2):\lambda\in\Lambda\},\\
            I_0&=24\times \{n\rho:n\in\Z_{>0}\},\\
            I_{<0}&=\varnothing,
        \end{align*}
        where $24\times \{n\rho:n\in\Z_{>0}\}$ means that each simple root $n\rho$ appears $24$ times in the datum. The Cartan matrix is the Gram matrix of this system.
        \item For any $\alpha\in II_{25,1}$, dimension of the corresponding root space is
        $$
        \dim \g_\alpha=p^{(24)}(1-\alpha^2/2).
        $$
    \end{enumerate}
\end{theorem}
\begin{proof}
    See \cite{Bor90}, Theorem 3 in particular. The first proof of $(iv)$ appeared in \cite{F85}.
\end{proof}

Clearly, the simple roots in $I_0$ are represented by $\tau(-1)\ket{n\rho}$, where $\tau\in \rho^\vee/\rho$, and the simple roots in $I$ are represented by $\ket\alpha$, where $\alpha=(\lambda,1,1-\lambda^2/2)$.
\begin{proposition}
    The epimorphism
    $$
    P_1/L_{-1}P_0\twoheadrightarrow\g
    $$
    splits. In other words, the elements $\tau(-1)\ket{\pm n\rho}$, $\ket{\pm\alpha}$, $n\rho(-1)$ and $\alpha(-1)$ already satisfy all the relations between $e_i,f_i,h_i$ of the BKM algebra $\g$.

    In particular, we can write $P_1/L_{-1}P_0=\g\ltimes \ker(\cdot,\cdot)$, where $(\cdot,\cdot)$ is the bilinear form described in Theorem \ref{t:fake_monster_properties}(ii).
\end{proposition}
\begin{proof}
    We have already had computations showing that Heisenberg and $\sl_2$ relations are satisfied.

    We now show the commutativities between positive and negative parts. Let $\alpha,\beta$ be two simple roots in $I_{>0}$. Then the Lie bracket $[\ket\alpha,\ket{-\beta}]$ lies in the root space $(P_1/L_{-1}P_0)_{\alpha-\beta}$. However,
    $$
    (\alpha-\beta)^2=\alpha^2+\beta^2-2(\alpha,\beta)\ge 4,
    $$
    which implies that $L_0$ acts on $(P_1/L_{-1}P_0)_{\alpha-\beta}$ by at least $2$. This is impossible since it is assumed to act by the scalar $1$. Thus, $[\ket\alpha,\ket{-\beta}]=0$.

    Similarly, $[\tau(-1)\ket{n\rho},\ket{-\beta}]$ lands into the $n\rho-\beta$ root space with
    $$
    (n\rho-\beta)^2=-n(\rho,\beta)+2=2+n>2,
    $$
    where the last equality follows from the form of $\beta$ and $\rho$. This shows that $[\tau(-1)\ket{n\rho},\ket{-\beta}]=0$, as desired.

    Now, we check the Serre relations, and the reason why they hold will be the same as before. Choose $\alpha,\beta\in I_{>0}$ and let $m:=-(\alpha,\beta)\ge 0$. Then
    $$
    (\alpha+(m+1)\beta)^2=2+2(m+1)^2-2m(m+1)=4+2m>2,
    $$
    showing the Serre relation between these two roots. Similarly,
    $$
    (n\rho+(n+1)\alpha)^2=2(n+1)^2-2n(n+1)=2+2n>2,
    $$
    showing the Serre relation between roots $\tau(-1)\ket{n\rho}$ and $\ket\alpha$. So, we are done.
\end{proof}
\begin{remark}
    One can show that 
    $$
    \dim (P_1/L_{-1}P_0)_\alpha=p^{(25)}(1-\alpha^2/2)-p^{(25)}(-\alpha^2/2)
    $$
    using the DDF construction (see \cite[Formula 3.50]{GN94}). So, the fake monster algebra is a relatively small subalgebra of $P_1/L_{-1}P_0$. 
\end{remark}

Now, notice that we have $24$ Heisenberg algebras with the same root spaces among the generators of $\g$. Therefore, we can deform them into Virasoro algebras and get VBKM algebras. Choose $T:=\{\tau_1,\ldots,\tau_{24}\}\subset II_{25,1}\otimes\C$ with $\tau_i^2=-1$, $(\tau_i,\rho)=-1$ for any $i$, and $\tau_i$ linearly independent. Let $\g_T$ be the Lie subalgebra of $V_1/TV_0$ generated by $\ket{\pm\alpha}$ for $\alpha\in I_{>0}$ and $\tau_i(-1)\ket{n\rho}$ for all $i=1\ldots24$ and $n\in\Z$. Note that the family of Lie subalgebras associated to $T(s,t):=\{s\tau_i+t\tau_i':i=1\ldots24\}$ connects these algebras and limits to the fake monster Lie algebra as $s\to 0$.

We get an immediate corollary of Theorems \ref{t:characters_of_g_in_families} and \ref{t:fake_monster_properties}.
\begin{theorem}
    For any $\alpha\in II_{25,1}$,  have an inequality
    $$
    p^{(24)}(1-\alpha^2/2)\ge \dim\g_{T,\alpha},
    $$
    with an equality satisfied for generic $T$.
\end{theorem}
In particular, the denominator of $\g_{T,\alpha}$ for generic $T$ is an automorphic form (see \cite{Bor95}), and is bounded by this automorphic form for any $T$.

\begin{example}
    Choose vectors $\tau_1',\ldots,\tau_{24}'\in \rho^\perp$ whose images modulo $\rho$ form a basis in $\rho^\perp/\rho$. Let
    $$
    \tau_i:=\tau_i'+(\tau_i'^2+2(\tau_i',\sigma))\rho+\sigma.
    $$
    The coefficients are chosen in such a way that $\tau_i^2=-1$ and $(\tau_i,\rho)=-1$. Note that $\tau_i$ are linearly independent since they are linearly-independent modulo $\sigma$.
\end{example}

\section{Appendix: the ultimate algebra $\tilde{\g}_{\Vir,\Heis}$}
Let $\mc F$ be an enhanced generalized Cartan datum. We define the ultimate algebra $\tilde{\g}_{\Vir,\Heis}(\mc F)$ by putting the algebras $\mathfrak{v}^j=\Vir\rtimes \Heis$ instead of $\Vir^j$. We prove Theorem \ref{t:tilde_g_Virasoro_facts} for this algebra.

It is enough to construct the representation $\widetilde M(\lambda)$ and prove Lemma \ref{l:tilde_g_Virasoro_representation} about it, after which the proof identically repeats the Virasoro case.

As in the Virasoro case, we define $G$ to be the free product of all $\g_i$ and $\vv^j$ and $G_-$ be the free product of all $\n_{i,-}$ and $\vv^j_{<0}$. Choose $\lambda\in \h^*$, and let 
    $$
    \widetilde M(\lambda):=U(G_-),
    $$
    which also equals the free product of $U(\n_{i,-})$ and $U(\vv^j_{<0})$. We will construct a representation of $G$ on $\widetilde M(\lambda)$ depending on $\lambda$ and show that it descends to a representation of $\tilde\g(\mc F)$. The generators of $G_-$ act by left multiplications. Elements of $\h$ act line in the Virasoro case with one additional relation
    $$
    [h,x_n^j]=n\gamma_j(h)x_n^j,\qquad n<0.
    $$
    The generators of $G_+$ act like in the Virasoro case, with additional rules
    \begin{align*}
        e_i.x_n^j&=0,\\
        L_m^j.x_n^{j_0}&=-n\delta_{j,j_0}x_{n+m}^j.1,\\
        x_m^{j}.f_i&=0,\\
        x_m^{j}.L_n^{j_0}&=m\delta_{j,j_0}x_{m+n}^j.1,\\
        x_m^j.x_n^{j_0}&=0,
    \end{align*}
    where $m\ge 0$, $n<0$.

    \begin{proof}
As in the Virasoro case, we first check that the operators preserve all relations in $G_-$, hence are well-defined. As before, left multiplication by elements in $G_-$ clearly preserves the relations in $G_-$. We must check the action of positive operators $L_k^{j_0}$ and $x_k^{j_0}$ for $k\ge 0$.

We have two additional relations in $G_-$ apart from the already checked ones in Lemma \ref{l:tilde_g_Virasoro_representation}:
\
$$
L_n^jx_m^j-x_m^jL_n^j=-mx_{n+m}^j,\qquad x_n^jx_m^j-x_m^jx_n^j=0
$$ 
for $n,m<0$. 

For the commutativity of the Heisenberg part, we compute for $k \ge 0$ and $m, n \le 0$:
\begin{align*}
[L_k^{j_0},x_n^jx_m^j-x_m^jx_n^j]&=[[L_k^{j_0},x_n^j],x_m^j]+[x_n^j,[L_k^{j_0},x_m^j]]\\
&=-n\delta_{j,j_0}[x_{k+n}^j,x_m^j]-m\delta_{j,j_0}[x_n^j,x_{k+m}^j]\\
&=\delta_{j,j_0}\delta_{n+m+k,0}(-n(k+n)-mn)\lambda(c_j)=0,
\end{align*}
where we apply induction on $k$.

For the action of Virasoro generators on the Virasoro-Heisenberg relations, for $k\ge 0$ and $m,n\le 0$ we compute:
\begin{align*}
[L_k^{j_0},L_n^jx_m^j&-x_m^jL_n^j+mx_{n+m}^j]=[[L_k^{j_0},L_n^j],x_m^j]+[L_n^j,[L_k^{j_0},x_m^j]]+m[L_k^{j_0},x_{n+m}^j]\\
&=\delta_{j,j_0}(k-n)[L_{k+n}^j,x_m^j]-\delta_{j,j_0}m[L_n^j,x_{k+m}^j]-\delta_{j,j_0}m(n+m)x_{k+n+m}^j\\
&=\delta_{j,j_0}\left(-(k-n)mx_{k+n+m}^j+m(k+m)x_{k+n+m}^j-m(n+m)x_{k+n+m}^j\right)\\
&=\delta_{j,j_0}(-km+nm+mk+m^2-mn-m^2)x_{k+n+m}^j=0,
\end{align*}
wherein the second equality we used the assumption that $c_j$ commutes with the Heisenberg generators.

For the action of Heisenberg generators on the Virasoro relations, we have for $k\ge 0$ and $m,n< 0$:
\begin{align*}
[x_k^{j_0},&L_n^jL_m^j-L_m^jL_n^j-(n-m)L_{m+n}^j]\\
&=[[x_k^{j_0},L_n^j],L_m^j]+[L_n^j,[x_k^{j_0},L_m^j]]-(n-m)[x_k^{j_0},L_{m+n}^j]\\
&=\delta_{j,j_0}k[x_{k+n}^j,L_m^j]+\delta_{j,j_0}k[L_n^j,x_{k+m}^j]-\delta_{j,j_0}(n-m)kx_{m+n+k}^j\\
&=\delta_{j,j_0}\left(k(k+n)x_{m+n+k}^j-k(k+m)x_{m+n+k}^j-(n-m)kx_{m+n+k}^j\right)=0.
\end{align*}
This finishes the proof that the operators are well-defined on $\widetilde M(\lambda)$.

Next, we check that these operators satisfy the relations of $\tilde\g_{\Vir, \Heis}(\mathcal{F})$. As in the proof of Lemma \ref{l:tilde_g_Virasoro_representation}, it is enough to only check the relations between positive or between negative generators. For negative generators, the relations are evident since $G_-$ acts by left multiplication. 

We need to check relations in $\vv^j$. For this, observe that if $A$ and $B$ are algebras with derivations $D_A$ and $D_B$ on them, then there exists a unique derivation $D$ on the free product $A*B$ extending $D_A$ and $D_B$. Then notice that $\vv^j_{>0}$ acts on $M(\lambda)\simeq U(G_-)$ by derivations extended from each free subproduct $U(\vv^{j_0}_{<0})$ and $U(f_i)$. The algebra $\vv^j_{>0}$ acts by $0$ on each of those algebas except $U(\vv^j_{<0})$, on which it acts as on its Verma module 
$$
M^j(\lambda):=U(\vv^j)\otimes_{U(\vv^j_{\ge 0})}\C_{\lambda}.
$$
This action satisfies all the relations in $U(\vv^j_{>0})$, so the extension does. Thus, the proof is complete.

\end{proof}
\section*{Acknowledgments}
I thank Professor Igor Frenkel for lots of fruitful discussions and his proposal to look at relations between Virasoro algebras inside lattice VOAs, which served as one of the motivating points for me to introduce VBKM algebras. I was partially supported by NSF grant DMS-2501558.

% \medskip
% \noindent \textbf{Funding:}

\section*{Declarations}
%\noindent \textbf{Conflict of Interest:} The author has no conflicts of interest to declare.

\noindent \textbf{AI usage:} AI was not used to produce any results in this paper or write any piece of this paper. 

\bibliographystyle{alphaurl}
%\nocite{*}
\bibliography{references}

@article{BGGN, 
author = {B{\"a}rwald, Oliver and Gebert, R. W. and G{\"u}naydin, Murat and Nicolai, Hermann}, 
title = {Missing modules, the {G}nome {L}ie algebra, and {E}10}, 
journal = {Communications in Mathematical Physics}, 
volume = {195}, 
year = {1998} 
}

@article{Bor88,
  title = {Generalized {K}ac-{M}oody algebras},
  volume = {115},
  ISSN = {0021-8693},
  DOI = {10.1016/0021-8693(88)90275-x},
  number = {2},
  journal = {Journal of Algebra},
  publisher = {Elsevier BV},
  author = {Borcherds,  Richard},
  year = {1988},
  pages = {501–512}
}

@article{Bor90,
  title = {The monster {L}ie algebra},
  volume = {83},
  ISSN = {0001-8708},
  DOI = {10.1016/0001-8708(90)90067-w},
  number = {1},
  journal = {Advances in Mathematics},
  publisher = {Elsevier BV},
  author = {Borcherds,  Richard E},
  year = {1990},
  pages = {30–47}
}

@article{Bor92,
author = {Borcherds, Richard E.},
journal = {Inventiones mathematicae},
number = {2},
pages = {405-444},
title = {Monstrous moonshine and monstrous {L}ie superalgebras.},
url = {http://eudml.org/doc/144026},
volume = {109},
year = {1992},
}

@article{Bor95,
author = {Borcherds, Richard E.},
journal = {Inventiones mathematicae},
number = {1},
pages = {161-214},
title = {Automorphic forms on {O(s + 2,2)(R)} and infinte products.},
url = {http://eudml.org/doc/144273},
volume = {120},
year = {1995},
}

@book{FBZ04,
  title = {Vertex Algebras and Algebraic Curves},
  ISBN = {9781470413156},
  ISSN = {2331-7159},
  DOI = {10.1090/surv/088},
  journal = {Mathematical Surveys and Monographs},
  publisher = {American Mathematical Society},
  author = {Frenkel,  Edward and Ben-Zvi,  David},
  year = {2004},
  month = Aug 
}

@incollection{F85,
  author       = {Frenkel, Igor B.},
  title        = {Representations of {K}ac-{M}oody Algebras and Dual Resonance Models},
  booktitle    = {Applications of Group Theory in Physics and Mathematical Physics},
  editor       = {M. Flato and P. Sally and G. Zuckerman},
  series       = {Lectures in Applied Mathematics},
  volume       = {21},
  pages        = {325--353},
  publisher    = {American Mathematical Society},
  address      = {Providence, RI},
  year         = {1985}
}

@book{FLM89,
  title={Vertex Operator Algebras and the {M}onster},
  author={Frenkel, I. and Lepowsky, J. and Meurman, A.},
  isbn={9780080874548},
  series={Pure and Applied Mathematics},
  year={1989},
  publisher={Academic Press}
}

@article{FT21,
  title = {The trascendence of {K}ac-{M}oody algebras},
  volume = {35},
  ISSN = {2406-0933},
  DOI = {10.2298/fil2110445f},
  number = {10},
  journal = {Filomat},
  publisher = {National Library of Serbia},
  author = {Fernández-Ternero,  Desamparados and Núñez-Valdés,  Juan},
  year = {2021},
  pages = {3445–3474}
}

@article{GN94, 
author = {Gebert, R. W. and Nicolai, H.}, 
title = {On {E}10 and the {DDF} construction}, 
journal = {Communications in Mathematical Physics}, 
volume = {166}, 
year = {1994} 
}

@misc{Jur96,
  title = {An exposition of generalized {K}ac-{M}oody algebras},
  ISBN = {9780821877852},
  ISSN = {0271-4132},
  DOI = {10.1090/conm/194/02391},
  journal = {Lie Algebras and Their Representations},
  publisher = {American Mathematical Society},
  author = {Jurisich,  Elizabeth},
  year = {1996},
  pages = {121–159}
}

@article{Jur98,
  title = {Generalized {K}ac-{M}oody {L}ie algebras,  free {L}ie algebras and the structure of the {M}onster {L}ie algebra},
  volume = {126},
  ISSN = {0022-4049},
  DOI = {10.1016/s0022-4049(96)00142-9},
  number = {1-3},
  journal = {Journal of Pure and Applied Algebra},
  publisher = {Elsevier BV},
  author = {Jurisich,  Elizabeth},
  year = {1998},
  pages = {233–266}
}

@book{Kac90,
  author    = {Kac, Victor G.},
  title     = {Infinite-Dimensional {L}ie Algebras},
  series    = {Cambridge Studies in Advanced Mathematics},
  volume    = {44},
  edition   = {3rd},
  publisher = {Cambridge University Press},
  address   = {Cambridge, UK},
  year      = {1990},
  isbn      = {0-521-37205-6},
}

@book{Kac98,
  title = {Vertex algebras for beginners},
  ISBN = {9780821813966},
  ISSN = {2376-919X},
  DOI = {10.1090/ulect/010},
  journal = {University Lecture Series},
  publisher = {American Mathematical Society},
  author = {Kac,  Victor},
  year = {1998}
}

@book{KR97,
author = {Kac, Victor G. and Raina, Ashok K.},
title = {{B}ombay Lectures On Highest Weight Representations Of Infinite Dimensional {L}ie Algebras},
publisher = {World Scientific},
year = {1987}}

@article{Nie02,
  title = {Some generalized {K}ac-{M}oody algebras with known root multiplicities},
  volume = {157},
  ISSN = {0065-9266},
  DOI = {10.1090/memo/0746},
  number = {746},
  journal = {Memoirs of the American Mathematical Society},
  publisher = {American Mathematical Society (AMS)},
  author = {Niemann,  Peter},
  year = {2002}
}

@book{Ray06,
  author = {Ray, Urmie}, 
  title = {Automorphic Forms and {L}ie Superalgebras},
  ISBN = {9781402050091},
  DOI = {10.1007/978-1-4020-5010-7},
  journal = {Algebra and Applications},
  publisher = {Springer Netherlands},
  year = {2006}
}

@article{Sha72,
  title = {On a bilinear form on the universal enveloping algebra of a complex semisimple {L}ie algebra},
  volume = {6},
  ISSN = {1573-8485},
  DOI = {10.1007/bf01077650},
  number = {4},
  journal = {Functional Analysis and Its Applications},
  publisher = {Springer Science and Business Media LLC},
  author = {Shapovalov,  N. N.},
  year = {1972},
  month = Oct,
  pages = {307–312}
}

\end{document}